\documentclass[12pt]{amsart}
\usepackage{amsmath,amsfonts,amssymb,amsthm, a4wide, url, mathscinet,mathdots,comment}
\usepackage{tikz-cd}
\usepackage{hyperref}

\newcommand{\Trans}{\mathcal{T}}
\newcommand{\Cycles}{\mathcal{C}}
\newcommand{\Special}{\Cycles^{2,3}}
\newcommand{\spcl}{\operatorname{spcl}}
\newcommand{\Pairs}{\mathcal{P}}
\newcommand{\Dyck}{\mathcal{F}}
\newcommand{\Adm}{\mathcal{A}}
\newcommand{\adm}{admissible}
\newcommand{\admis}{admissibility}
\newcommand{\SEss}{\mathfrak{E}}
\newcommand{\dcr}{decoration}
\newcommand{\dcrt}{decorated}
\newcommand{\dbi}{defined by inclusions}
\newcommand{\prtt}{\mathfrak{X}}
\newcommand{\prtset}{X}
\newcommand{\grph}{\mathcal{G}}
\newcommand{\cycr}{R}
\newcommand{\cycl}{L}
\newcommand{\Ess}{\mathcal{E}}
\newcommand{\tbl}{\mathbf{C}}
\newcommand{\tblc}{\mathbf{D}}
\newcommand{\OPairs}{\mathcal O}
\newcommand{\EPairs}{\mathcal E}
\newcommand{\smth}{{\operatorname{sm}}}
\newcommand{\precdot}{\prec\mathrel{\mkern-5mu}\mathrel{\cdot}}
\newcommand{\adj}{\precdot}
\newcommand{\rshft}[2]{\cycr_{[{#1},{#2}]}}
\newcommand{\lshft}[2]{\cycl_{[{#1},{#2}]}}
\newcommand{\maxperm}[1]{\mu_{#1}}

\newtheorem{theorem}{Theorem}[section]
\newtheorem{lemma}[theorem]{Lemma}
\newtheorem{claim}[theorem]{Claim}
\newtheorem{definition}[theorem]{Definition}

\newtheorem{proposition}[theorem]{Proposition}
\newtheorem{corollary}[theorem]{Corollary}
\theoremstyle{remark}
\newtheorem{remark}[theorem]{Remark}
\newtheorem{example}[theorem]{Example}
\newtheorem{question}[theorem]{Question}
\newtheorem{observation}[theorem]{Observation}

\numberwithin{equation}{section}

\begin{document}

\title{Some combinatorial results on smooth permutations}

\author{Shoni Gilboa}
\address{Department of Mathematics and computer Science, The Open University of Israel, Raanana 4353701, Israel}

\author{Erez Lapid}
\address{Department of Mathematics, Weizmann Institute of Science, Rehovot 7610001, Israel}
\email{erez.lapid@weizmann.ac.il}

\begin{abstract}
We show that any smooth permutation $\sigma\in S_n$ is characterized by the set $\tbl(\sigma)$ of transpositions and $3$-cycles
in the Bruhat interval $(S_n)_{\leq\sigma}$, and that $\sigma$ is the product (in a certain order) of the transpositions
in $\tbl(\sigma)$.
We also characterize the image of the map $\sigma\mapsto\tbl(\sigma)$.
As an application, we show that $\sigma$ is smooth if and only if the intersection of $(S_n)_{\leq\sigma}$
with every conjugate of a parabolic subgroup of $S_n$ admits a maximum.
This also gives another approach for enumerating smooth permutations and subclasses thereof.
Finally, we relate covexillary permutations to smooth ones and rephrase the results in terms of the (co)essential set
in the sense of Fulton.
\end{abstract}


\maketitle

\setcounter{tocdepth}{1}
\tableofcontents

\section{Introduction}\label{sec:intro}

\subsection{}
Fix an integer $n\geq1$.
Consider the symmetric group $S_n$ of all the permutations of the set $[n]=\{1,2,\ldots,n\}$,
with the Bruhat order $\leq$.
Let $\Trans=\{T_{i,j}:1\leq i<j\leq n\}\subset S_n$ be the set of transpositions.
For every permutation $\sigma\in S_n$ define the \emph{2-table} of $\sigma$ to be
\[
\tbl_\Trans(\sigma)=\{\tau\in\Trans:\tau\leq\sigma\}.
\]
For every $\sigma\in S_n$ we have $\ell(\sigma)\leq\# \tbl_\Trans(\sigma)$ where
\[
\ell(\sigma)=\#\{i<j:\sigma(i)>\sigma(j)\}
\]
is the number of inversions of $\sigma$ \cite{MR752799}.
If $\ell(\sigma)=\# \tbl_\Trans(\sigma)$, then $\sigma$ is called \emph{smooth},
a terminology that is justified by the fact that this condition also characterizes the smoothness
of the Schubert variety $X_\sigma$ pertaining to $\sigma$ [ibid.].
Another well-known combinatorial characterization of the smoothness of $\sigma$ is that
$\sigma$ is $4231$ and $3412$ avoiding \cite{MR1051089}.
We refer to \cite{MR1782635} and the references therein for more information about singularities of Schubert varieties.

Distinct smooth permutations may have the same $2$-table (for example, for $n=3$,
$\tbl_\Trans((231))=\{T_{1,2},T_{2,3}\}=\tbl_\Trans((312))$).
However, we show that smooth permutations are distinguishable from each other at the `next level'.
More precisely, let $\Special\subset S_n$ be the set of permutations consisting of a single cycle of length $2$ or $3$.
Denote the $3$-cycle permutation $i\mapsto j\mapsto k\mapsto i$ with $i<j<k$ by $R_{i,j,k}$, so that
\[
\Special=\Trans\cup\{R_{i,j,k},R_{i,j,k}^{-1}:i<j<k\}.
\]
We define the \emph{$2$-$3$-table} of a permutation $\sigma\in S_n$ to be
\[
\tbl(\sigma)=\{\tau\in\Special:\tau\leq\sigma\}.
\]
Clearly, $\tbl(\sigma)$ is downward closed and it is easy to see that
if $R_{i,j,l},R_{i,k,l}^{-1}\in\tbl(\sigma)$ with $i<j,k<l$, then $T_{i,l}\in\tbl(\sigma)$.

Our main result is the following.
\begin{theorem}\label{thm:bijection} [See \S\ref{sec:bijection}]
The map $\sigma\mapsto\tbl(\sigma)$ defines a bijection between the smooth permutations of $S_n$
and the downward closed subsets $A$ of $\Special$ that satisfy the following two conditions.
\begin{itemize}
\item If $R_{i,j,l},R_{i,k,l}^{-1}\in A$ with $i<j,k<l$, then $T_{i,l}\in A$.
\item Whenever $T_{i,j},T_{j,k}\in A$, $i<j<k$, at least one of $R_{i,j,k}$ and $R_{i,j,k}^{-1}$ belongs to $A$.
\end{itemize}
The inverse bijection $A\mapsto\pi(A)$ is given by
\begin{equation} \label{eq: pialt}
\pi(A)=\max\{\tau\in S_n:\tbl(\tau)=A\}=\max\{\tau\in S_n:\tbl_\Trans(\tau)=A_\Trans,\,\tbl(\tau)\subseteq A\}
\end{equation}
where the maximum (i.e., the greatest element, which in particular exists) is with respect to the Bruhat order.
\end{theorem}

\subsection{}
The subsets $A\subset\Special$ satisfying the properties of Theorem \ref{thm:bijection} will be called \emph{\adm}.
We give an alternative, more constructive definition of $\pi(A)$ for an {\adm} set $A\subset\Special$.
We say that a linear order $\prec$ on $A_\Trans=A\cap\Trans$ is \emph{compatible} (with $A$) if whenever $T_{i,j},T_{j,k}\in A$, $i<j<k$:
\begin{itemize}
\item If $T_{i,k}\in A$, then either $T_{i,j}\prec T_{i,k}\prec T_{j,k}$ or $T_{j,k}\prec T_{i,k}\prec T_{i,j}$.
\item If $T_{i,k}\notin A$, then $R_{i,j,k}\in A\iff T_{i,j}\prec T_{j,k}$.
\end{itemize}

\begin{theorem} \label{rek} [See \S\S\ref{sec:orders}, \ref{sec:bijection}]
Let $A$ be an {\adm} subset of $\Special$. Then, a compatible order on $A_\Trans$ always exists
and $\pi(A)$ is equal to the product of the elements of $A_\Trans$ taken with respect to a compatible order $\prec$.
(In particular, the product depends only on $A$.)
Consequently, every smooth permutation may be written as the product, in an appropriate order,
of the transpositions in its $2$-table (each appearing exactly once).
\end{theorem}

More precisely, we define a graph $\grph_A$ whose vertices are the compatible orders on $A_\Trans$
and whose edges connect two compatible orders that can be obtained from one another by one of the following elementary operations.
\begin{itemize}
\item Interchanging the order of two adjacent commuting transpositions.
\item Switching the order of consecutive $T_{i,j}$, $T_{i,k}$, $T_{j,k}$ (with $i<j<k$) to $T_{j,k}$, $T_{i,k}$, $T_{i,j}$,
or vice versa.
\end{itemize}
These operations do not change the product of the elements of $A_\Trans$, taken in the respective orders.
In Section \ref{sec:orders} we show that $\grph_A$ is connected (and in particular, non-empty).
In other words, a compatible order exists and every two compatible orders are obtained from one another
by a sequence of elementary operations. The situation is reminiscent of the case of reduced decompositions
of a permutation $\sigma$, which form the vertices of a connected graph $G(\sigma)$ whose edges are given by basic Coxeter relations.
In fact, for $A=\Special$ itself, there is a natural isomorphism between $\grph_A$ and $G(w_0)$ where
$w_0$ is the longest permutation.
However, for a general smooth permutation $\sigma$, the number of compatible orders on $\tbl_\Trans(\sigma)$
with respect to $\tbl(\sigma)$ does not agree with the number of reduced decompositions of $\sigma$.

\subsection{}
As an application of Theorem \ref{thm:bijection} we give another remarkable characterization of smooth permutations.
Let $\prtt$ be a partition (i.e., an equivalence relation) of $[n]$.
Consider the subgroup $S_{\prtt}$ of $S_n$ preserving every $\prtset\in\prtt$.
The group $S_{\prtt}$ is isomorphic to the direct product of $S_{\# \prtset}$ over $\prtset\in \prtt$.
The product order on $S_{\prtt}$, which we denote by $\leq_{\prtt}$, is (in general, strictly) stronger than the one induced from $S_n$.
We say that an element of $S_{\prtt}$ is \emph{$\prtt$-smooth} if all its coordinates in $S_{\# \prtset}$, $\prtset\in \prtt$ are smooth.
(This is weaker than smoothness in $S_n$.)

\begin{theorem} \label{thm: bruintH} [See \S\ref{sec:misc}]
$\sigma\in S_n$ is smooth if and only if for every partition $\prtt$ of $[n]$, the set
\[
\{\tau\in S_{\prtt}:\tau\leq\sigma\}
\]
admits a maximum $\sigma_{\prtt}$ with respect to $\leq_{\prtt}$.
Moreover, in this case $\sigma_{\prtt}$ is $\prtt$-smooth.
\end{theorem}

\subsection{}
We may also interpret the bijection of Theorem \ref{thm:bijection} in terms of more familiar
combinatorial objects, namely Dyck paths.
We may view a Dyck path as a weakly increasing function $f:[n]\to[n]$ such that $f(i)\geq i$ for all $i$.
(Their number is the Catalan number $C_n=\frac1{n+1}{\binom{2n}n}$.)
By definition, a \emph{\dcrt}\ Dyck path is such a function $f$ together with a function $g:[n]\rightarrow\{0,1\}$ such that
\begin{enumerate}
\item $g(i)=0$ whenever $f(f(i))=f(i)$.
\item $g(i)=g(i+1)$ whenever $i<n$ and $f(i+1)<f(f(i))$.
\end{enumerate}

In terms of Dyck paths, such a \dcr\ $g$ corresponds to an (unrestricted) $2$-coloring of a certain distinguished set of vertices of the path.

Write $g^{-1}(0)=\{i_1,\ldots,i_k\}$ and  $g^{-1}(1)=\{j_1,\ldots,j_l\}$ with $i_1<\cdots<i_k$ and $j_1<\cdots<j_l$.

For every $1\leq i<j\leq n$ let $\rshft ij\in S_n$ be the cycle permutation $i\to i+1\to\cdots\to j\to i$.
For consistency, $R_{[i,i]}$ is the identity permutation for all $i$.

\begin{theorem}\label{thm:sigma_fg_bijection} [See \S\ref{sec:sigma_fg_bijection}]
The map
\begin{equation}  \label{eq: fgbij}
(f,g)\to(\rshft{j_l}{f(j_l)}\cdots\rshft{j_1}{f(j_1)})^{-1}\rshft{i_k}{f(i_k)}\cdots\rshft{i_1}{f(i_1)}
\end{equation}
is a bijection between the \dcrt\ Dyck paths and the smooth permutations in $S_n$.
Moreover, the expression on the right-hand side of \eqref{eq: fgbij} is reduced.
\end{theorem}

Theorem \ref{thm:sigma_fg_bijection} is in the spirit of Skandera's factorization of smooth permutation  \cite{MR2387587}.
Using Theorem \ref{thm:sigma_fg_bijection}, we can recover several known enumerative results concerning smooth permutations.

\subsection{}
Using Theorem \ref{thm:bijection}, we can also give a curious relation between smooth
permutations and \emph{covexillary} ones.
Recall that a permutation is called covexillary if it avoids the pattern $3412$.

\begin{theorem}\label{thm: idemp} [See \S\ref{sec:retract}]
For any covexillary $\tau\in S_n$, $\tbl(\tau)$ is {\adm}.
Therefore, the map $\tau\mapsto\pi(\tbl(\tau))$ is an idempotent function from
the set of covexillary permutations onto the subset of smooth permutations.
Moreover, this map is order preserving and for any covexillary $\tau\in S_n$,
\[
\pi(\tbl(\tau))=\min\{\sigma\in S_n\text{ smooth}:\sigma\geq\tau\}.
\]
\end{theorem}

\subsection{}

Finally, we can relate our results to Fulton's notion of essential set \cite{MR1154177}.
For any $\sigma\in S_n$ let $\Ess(\sigma)$ be the set of pairs $(i,j)\in [n-1]\times[n-1]$ such that
\[
\sigma(i)\le j<\sigma(i+1)\text{ and }\sigma^{-1}(j)\le i<\sigma^{-1}(j+1).
\]
Up to change of coordinates, this is the essential set of $w\sigma_0$ in Fulton's formulation.
\begin{theorem} \label{thm: essset} [See \S\ref{sec: ess}]
The map $\sigma\mapsto\Ess(\sigma)$ defines a bijection between the smooth permutations in $S_n$ and
the subsets $E\subseteq [n-1]\times[n-1]$ satisfying the property that for every two distinct points
$(i_1,j_1)$ and $(i_2,j_2)$ in $E$ such that $\min(i_2,j_2)\ge\min(i_1,j_1)$ we have
\[
i_2\ge i_1,\ \ j_2\ge j_1,\ \ \max(i_2,j_2)>\max(i_1,j_1)\text{ and }\min(i_2,j_2)>\min(i_1,j_1).
\]
\end{theorem}
We also relate this result to the previous theorems.

\subsection{}

Although we will not discuss it any further here, we mention that smooth permutations are important in representation theory.
This is because of the \emph{Kazhdan--Lusztig conjecture} \cite{MR560412}
(proved independently by Bernstein--Beilinson and Brylinski--Kashiwara) and the fact that smooth permutations are
characterized in terms of \emph{Kazhdan--Lusztig polynomials} \cite{MR788771}.
See \cite{MR3866895} for a more recent surprising occurrence of smooth permutations in representation theory.

Theorem \ref{thm: bruintH} was the original motivation of the paper.
It came up in studying a related problem, which is discussed in \cite{MR3880654}.
The result of [ibid.] is relevant for a certain representation-theoretic context.
We hope that the same will be true for Theorem \ref{thm: bruintH} and its variants,
although we will not discuss these possible applications here.

Likewise, it would be interesting to find a geometric context for Theorems \ref{thm:bijection} and \ref{thm: bruintH}.

It is natural to ask whether Theorem \ref{thm:bijection} admits an analogue for other Weyl groups $W$.
In particular, one may ask whether any smooth element $w$ of $W$ can be written as the product (in a suitable order)
of the reflexions that are smaller than or equal to $w$ in the Bruhat order (each reflexion occurring exactly once).

\medskip
The rest of the paper is organised as follows.
In Section \ref{sec:adm} we study the notion of {\adm} sets.
In Section \ref{sec:wedges} we develop some tools that will enable us to apply inductive arguments.
In Section \ref{sec:orders} we study the notion of a compatible order.
Sections \ref{sec:bijection}, \ref{sec:misc}, \ref{sec:sigma_fg_bijection}, \ref{sec:retract} and \ref{sec: ess} are devoted to
the proofs of Theorems \ref{thm:bijection}, \ref{thm: bruintH}, \ref{thm:sigma_fg_bijection}, \ref{thm: idemp} and \ref{thm: essset}, respectively. In Section \ref{sec:enumeration} we use Theorem \ref{thm:sigma_fg_bijection} to reproduce some
known enumerative results concerning smooth permutations.

\smallskip
An extended abstract of this paper appears in \cite{ShoniProc}.

\subsection{Notation and preliminaries}

Recall that the Bruhat order on $S_n$ is defined by inclusion of Schubert varieties.
(See \cite{MR2133266} for standard facts about the Bruhat order.)
It can be described combinatorially as the partial order generated by
\[
\sigma<\sigma T_{i,j}\text{ if }\sigma(i)<\sigma(j).
\]
It is also be given by
\begin{equation}\label{eq:Bruhat}
\tau\leq\sigma\text{ if and only if }\#(\tau([i])\cap[j])\geq\#(\sigma([i])\cap[j]) \text{ for every }  i,j.
\end{equation}
This relation endows $S_n$ with the structure of a ranked poset, with rank function $\ell$.
The minimum of $S_n$ is the identity permutation $e$ and the
maximum is the permutation $w_0$ given by $w_0(i)=n+1-i$, $i\in[n]$.

The Bruhat order is invariant under inversion $\sigma\mapsto\sigma^{-1}$ and under upending $\sigma\mapsto w_0\sigma w_0$.

For any $\sigma\in S_n$ we denote by $(S_n)_{\leq\sigma}$ the Bruhat interval defined by $\sigma$.

For any $\sigma\in S_n$ define the ``maximal function'' $\maxperm{\sigma}:[n]\rightarrow[n]$ of $\sigma$ by
\[
\maxperm{\sigma}(i)=\max\sigma([i]).
\]
Clearly, $\maxperm\sigma(i)\geq i$ with equality if and only if $\sigma([i])=[i]$.
Also, if $\tau\leq\sigma$ then $\maxperm{\tau}\leq\maxperm{\sigma}$ pointwise, although the converse in not true in general.

Recall that $\Trans=\{T_{i,j}\}_{1\leq i<j\leq n}$ where $T_{i,j}$ is the transposition interchanging $i$ and $j$.
These are the reflexions of $S_n$, as a Coxeter group.

We denote by $(S_n)_\smth$ the set of smooth permutations.

Recall that $R_{i,j,k}$, $1\leq i<j<k\leq n$ is the $3$-cycle permutation $i\mapsto j\mapsto k\mapsto i$.
Let also $L_{i,j,k}=R_{i,j,k}^{-1}$.
We have $w_0R_{i,j,k}w_0=L_{w_0(k),w_0(j),w_0(i)}$ and $w_0L_{i,j,k}w_0=R_{w_0(k),w_0(j),w_0(i)}$.
We write
\[
\Special=\Special_n=\Trans\cup\{R_{i,j,k},L_{i,j,k}\}_{1\leq i<j<k\leq n}
\]
and for any $\sigma\in S_n$
\[
\tbl(\sigma)=\{\tau\in\Special:\tau\leq\sigma\}.
\]
Note that $\tbl(\sigma^{-1})=\tbl(\sigma)^{-1}$ and $\tbl(w_0\sigma w_0)=w_0\tbl(\sigma)^{-1}w_0$.

It is useful to bear in mind the following explication of the Bruhat order for $S_n$.
\begin{subequations} \label{eq: easy1<=}
\begin{align} \label{eq: easy1<=a}
T_{i,j}\leq\sigma&\iff\maxperm{\sigma}(i)\geq j\text{ and }\maxperm{\sigma^{-1}}(i)\geq j,\\
R_{i,j,k}\leq\sigma&\iff\maxperm{\sigma}(i)\geq j,\,\maxperm{\sigma}(j)\geq k\text{ and }\maxperm{\sigma^{-1}}(i)\geq k,\\
L_{i,j,k}\leq\sigma&\iff\maxperm{\sigma^{-1}}(i)\geq j,\,\maxperm{\sigma^{-1}}(j)\geq k\text{ and }\maxperm{\sigma}(i)\geq k.
\end{align}
\end{subequations}
In particular,
\begin{subequations} \label{eq:for_wedge1<=}
\begin{align}
T_{i-1,i}\leq\sigma&\iff\maxperm\sigma(i)>i\iff\maxperm{\sigma^{-1}}(i)>i\iff\sigma([i-1])\ne[i-1], \label{eq:wedgeT}\\
R_{i,j,j+1}\leq\sigma&\iff\maxperm{\sigma}(i)\geq j\text{ and }\maxperm{\sigma^{-1}}(i)\geq j+1. \label{eq:wedgeR}
\end{align}
\end{subequations}
It also follows that for any $ i<j,k<l$,
\begin{equation}\label{eq: ijkl}
R_{i,k,l}\vee L_{i,j,l}=T_{i,l}\text{, i.e., for every }\sigma\in S_n:
\sigma\geq R_{i,k,l},L_{i,j,l}\iff\sigma\geq T_{i,l}.
\end{equation}

The comparisons among the elements of $\Special$ with respect to the Bruhat order are summarized in the following list.
\begin{subequations} \label{eq: easy<=}
\begin{align}
T_{i,j}\leq T_{x,y}\iff& x\leq i<j\leq y, \label{eq: east<=a}\\
R_{i,j,k}\leq T_{x,y}\iff L_{i,j,k}\leq T_{x,y}\iff& T_{i,k}\leq T_{x,y},\\
T_{i,j}\leq R_{x,y,z}\iff T_{i,j}\leq L_{x,y,z}\iff& \text{either }T_{i,j}\leq T_{x,y}\text{ or }T_{i,j}\leq T_{y,z},\\
R_{i,j,k}\leq L_{x,y,z}\iff L_{i,j,k}\leq R_{x,y,z}\iff&\text{either }T_{i,k}\leq T_{x,y}\text{ or }T_{i,k}\leq T_{y,z},\\
R_{i,j,k}\leq R_{x,y,z}\iff L_{i,j,k}\leq L_{x,y,z}\iff& \begin{aligned}&\text{either }T_{i,k}\leq T_{x,y}\text{ or }T_{i,k}\leq T_{y,z}\\
&\text{or }x\leq i<j=y<k\leq z.\end{aligned}
\end{align}
\end{subequations}

\section{Admissible sets}\label{sec:adm}

\subsection{}
\begin{definition}\label{def:adm}
We say that a subset $A\subseteq\Special$ is \emph{{\adm}} if it satisfies the following three conditions.
\begin{subequations}
\begin{align}
\label{item: adm1} &\text{$A$ is downward closed, i.e., if }\sigma\in A, \tau\in\Special \text{ and }\tau\leq\sigma,\text{ then }\tau\in A.\\
\label{item: adm2} &\text{If }R_{i,j,l},L_{i,k,l}\in A \text{ for some $i<j,k<l$, then }T_{i,l}\in A.\\
\label{item: adm3} &\text{If }T_{i,j},T_{j,k}\in A \text{ for some $i<j<k$, then at least one of }R_{i,j,k} \text{ or }L_{i,j,k} \text{ is in }A.
\end{align}
\end{subequations}
\end{definition}

\smallskip
Note that by \eqref{eq: easy<=}, the first two conditions imply that for every $i<j<k$,
\[
\text{$T_{i,k}\in A$ if and only if $R_{i,j,k}\in A$ and $L_{i,j,k}\in A$}.
\]

It is clear that if $A$ is {\adm}, then so are the inverted set $A^{-1}=\{\sigma^{-1}:\sigma\in A\}$ and
the upended set $w_0Aw_0=\{w_0\sigma w_0:\sigma\in A\}$.

We verify the first part of Theorem \ref{thm: idemp} in the following lemma.
\begin{lemma} \label{lem: covexisadm}
If $\tau\in S_n$ is covexillary, then $\tbl(\tau)$ is {\adm}.
\end{lemma}

\begin{proof}
By transitivity of the Bruhat order and \eqref{eq: ijkl},
the set $A=\tbl(\tau)$ satisfies properties \eqref{item: adm1} and \eqref{item: adm2} for every $\tau\in S_n$.

To prove \eqref{item: adm3} for covexillary $\tau$, assume on the contrary that $T_{i,j},T_{j,k}\in A$ but $R_{i,j,k},L_{i,j,k}\notin A$.
Then,
\begin{itemize}
\item $\maxperm{\tau}(i)\geq j,\,\maxperm{\tau^{-1}}(i)\geq j$,
\item $\maxperm{\tau}(j)\geq k,\,\maxperm{\tau^{-1}}(j)\geq k$,
\item $\maxperm{\tau}(i)<k,\,\maxperm{\tau^{-1}}(i)<k$.
\end{itemize}
Therefore, there exist $u,v,x,y\in [n]$ such that
\begin{itemize}
\item $u\leq i$ and $j\leq\tau(u)<k$,
\item $i<v\leq j$ and $\tau(v)\geq k$,
\item $j\leq x<k$ and $\tau(x)\leq i$,
\item $y\geq k$ and $i<\tau(y)\leq j$.
\end{itemize}
Thus, $u<v\leq x<y$ and $\tau(x)<\tau(y)\leq\tau(u)<\tau(v)$ in contradiction to the assumption on $\tau$.
\end{proof}

\begin{remark}
The assumption on $\tau$ in Lemma \ref{lem: covexisadm} is essential.
For instance, for $n=4$ and $\tau=(3412)$ we have
$$
\tbl(\tau)=\Special\setminus\{T_{1,4},R_{1,2,4},L_{1,2,4},R_{1,3,4},L_{1,3,4}\}
$$
which is not {\adm} since \eqref{item: adm3} is not satisfied.

On the other hand, the converse to Lemma \ref{lem: covexisadm} is also false.
For $n=5$ and $\tau=(45231)$ we have
$$
\tbl(\tau)=\Special\setminus\{T_{1,5},L_{1,2,5},L_{1,3,5},L_{1,4,5}\}
$$
which is {\adm} although $\tau$ is not covexillary.
\end{remark}

\subsection{}
The following observation follows directly from \eqref{eq: easy1<=a}.
\begin{observation}\label{obs:sub_nsub}
Let $\tau\in S_n$, $i<j$ and $x<y$. Then,
\begin{enumerate}
\item If $\maxperm{\tau}(x)<y\leq\maxperm{\tau T_{i,j}}(x)$, then $i\leq x<j$ and $\tau(i)<y\leq\tau(j)$.
\item If $\maxperm{\tau^{-1}}(x)<y\leq\maxperm{(\tau T_{i,j})^{-1}}(x)$, then $i<y\leq j$ and $\tau(i)\leq x<\tau(j)$.
\end{enumerate}
\end{observation}
For inductive arguments, the following result will be useful.
\begin{lemma} \label{lem: indsmooth}
Suppose that $\sigma\in S_n$ and $i\in [n]$ are such that $\sigma([i-1])=[i-1]$ and $\sigma(i)\geq j:=\sigma^{-1}(i)>i$.
(In particular, $\sigma(j)=i<\sigma(j-1)$.) Let $\sigma'=\sigma T_{j-1,j}$. Then,
\begin{equation} \label{eq: tblcntn}
\tbl_\Trans(\sigma')\subseteq\tbl_\Trans(\sigma)\text{ and }T_{i,j}\in\tbl_\Trans(\sigma)\setminus\tbl_\Trans(\sigma')
\end{equation}
and
\begin{equation} \label{eq: tblll}
\tbl_\Trans(\sigma')\supseteq\{T_{r,s}\in\tbl_\Trans(\sigma):r\ne j-1\text{ and }s\ne j\}.
\end{equation}
Moreover, $\sigma$ is smooth if and only if $\sigma'$ is smooth and $\tbl_\Trans(\sigma)=\tbl_\Trans(\sigma')\cup\{T_{i,j}\}$.
\end{lemma}

\begin{proof}
It is clear from the assumptions that $\sigma'<\sigma$ and $T_{i,j}\in\tbl_\Trans(\sigma)\setminus\tbl_\Trans(\sigma')$.
Hence, \eqref{eq: tblcntn}.

The inclusion \eqref{eq: tblll} follows from Observation \ref{obs:sub_nsub} and \eqref{eq: easy1<=a}.

Consider the second part.
Note that by our assumptions, $\ell(\sigma')=\ell(\sigma)-1$.
Since $\#\tbl_\Trans(\sigma')\geq\ell(\sigma')$, it follows from \eqref{eq: tblcntn} that the two conditions are equivalent.
\end{proof}

\begin{corollary}
Let $e\ne\sigma\in (S_n)_\smth$. Then, there exists $\sigma'\in (S_n)_\smth$ such that
$\ell(\sigma')=\ell(\sigma)-1$ and at least one of $\sigma^{-1}\sigma'$ or $\sigma'\sigma^{-1}$ is a simple reflection.
\end{corollary}

\begin{proof}
Let $i$ be the minimal non-fixed point of $\sigma$. Replacing $\sigma$ by $\sigma^{-1}$ if necessary, we may assume that
$j:=\sigma^{-1}(i)\leq \sigma(i)$.
Note that $j>i$.
Then $\sigma':=\sigma T_{j-1,j}$ is smooth by Lemma \ref{lem: indsmooth}, $\sigma^{-1}\sigma'=T_{j-1,j}$ is a simple reflection and $\ell(\sigma')=\ell(\sigma)-1$.
\end{proof}

\begin{example}
The permutation $\sigma=(365214)$ is smooth but there is no $i\in[5]$ such that $\sigma(i)>\sigma(i+1)$ and $\sigma T_{i,i+1}$ is smooth.
\end{example}

\section{Wedges and derived sets}\label{sec:wedges}

In this section we define a few technical notions that are useful for future inductive proofs.

\subsection{}
Given a subset $A\subseteq\Special$ we will write for simplicity $A_\Trans=A\cap\Trans$.
Note that $(\tbl(\sigma))_\Trans=\tbl_\Trans(\sigma)$.

\begin{observation}\label{obs:adm_subset}
If $A_1\subseteq A_2$ are {\adm} sets and $(A_1)_\Trans=(A_2)_\Trans$, then necessarily $A_1=A_2$.
\end{observation}

\begin{definition} \label{def: wedge}
Suppose that $A\subseteq\Special$ is {\adm} and $T=T_{i,j}\in A_\Trans$.
Then, we say that $T$ is a wedge for $A$ if $T_{i-1,i}\notin A$ (or $i=1$) and $R_{i,j,j+1}\notin A$ (or $j=n$).
\end{definition}

\begin{observation}\label{obs:derived_wedges}
If $T_{i,j}$ is a wedge for $A$, then $\{T\in A_\Trans: T(i)\neq i\}=\{T_{i,r}\}_{r=i+1}^j$.
\end{observation}

\begin{remark} \label{rem: existwedge}
Note that the definition of wedge is not symmetric with respect to either $A\mapsto A^{-1}$, $A\mapsto w_0Aw_0$, or $A\mapsto w_0A^{-1}w_0$.
In other words, if $T$ is a wedge for $A$, then $T$ is not necessarily a wedge for $A^{-1}$,
nor is $w_0Tw_0$ a wedge for either $w_0Aw_0$ or $w_0A^{-1}w_0$ in general.
However, for every non-empty {\adm} $A$, at least one of $A$ or $A^{-1}$ has a wedge.
Namely, if $T_{i,j}\in A_\Trans$ with $i$ minimal and $j$ maximal (with respect to this $i$), then $T_{i,j}$ is a wedge
for $A$ or $A^{-1}$ (or both).
\end{remark}

We record some simple properties of wedges in the following lemma.
\begin{lemma} \label{lem: simplewedge}
Suppose that $A\subseteq\Special$ is {\adm} and $T_{i,j}$ is a wedge for $A$.
Then,
\begin{enumerate}
\item For every $k>j$, $L_{i,j,k}\in A$ if and only if $T_{j,k}\in A$.
\item \label{item: RiklnotinA}
$R_{i,k,l}\notin A$ for all $i<k\leq j<l$.
\end{enumerate}
\end{lemma}

\begin{proof}
Suppose that $k>j$. If $L_{i,j,k}\in A$, then $T_{j,k}\in A$, by \eqref{eq: easy<=} and \eqref{item: adm1}.
If $T_{j,k}\in A$, then $L_{i,j,k}\in A$ or $R_{i,j,k}\in A$, by \eqref{item: adm3}, but $R_{i,j,k}\notin A$ by \eqref{eq: easy<=}
and \eqref{item: adm1}, since $R_{i,j,j+1}\notin A$. This proves the first part.
The second part holds since otherwise we would have $R_{i,k,j+1}\in A$ in contradiction to the property \eqref{item: adm2} for $A$
and the fact that $L_{i,j,j+1}\in A$ (by the first part).
\end{proof}

\subsection{}
\begin{definition} \label{def: derived set}
Suppose that $A\subseteq\Special$ is {\adm} and $T=T_{i,j}$ is a wedge for $A$.
Then, the derived set of $A$ with respect to $T$ is
$$
A':=A\setminus\big(\{T\}\cup\{L_{i,j,k}:k>j\}\cup\{R_{i,k,j}:i<k<j\}\big).
$$
In particular,
$$
A'_{\Trans}=A_\Trans\setminus\{T\}.
$$
\end{definition}

We record some simple properties of derived sets in the following lemma.
\begin{lemma} \label{lem: simplederived}
Suppose that $A\subseteq\Special$ is {\adm} and $T=T_{i,j}$ is a wedge for $A$.
Then,
\begin{enumerate}
\item \label{item:recover}
$A=A'\cup\{T\}\cup\{L_{i,j,k}:T_{j,k}\in A'\}\cup\{R_{i,k,j}:i<k<j\}$.
In particular, $A$ is determined by $T$ and $A'$.
\item \label{item: LkjlinA'}
$L_{k,j,l}\in A'$ for every $i<k<j<l$ such that $T_{j,l}\in A$.
\item $A'$ is {\adm}.
\item If $j>i+1$, then $T_{i,j-1}$ is a wedge for $A'$.
\end{enumerate}
\end{lemma}

\begin{proof}
The first two parts follow easily from the first part of Lemma \ref{lem: simplewedge}.

Write $A'=A\setminus B$ where
\[
B=\{T\}\cup\{L_{i,j,k}:k>j\}\cup\{R_{i,k,j}:i<k<j\}.
\]
To check that $A'$ is {\adm}, we check the properties \eqref{item: adm1}--\eqref{item: adm3}.

For \eqref{item: adm1} we need to check that if $\tau\leq\sigma$ for some $\tau\in B$ and $\sigma\in\Special$,
then $\sigma\in B$ or $\sigma\notin A$. This follows from the table \eqref{eq: easy<=} and the second part of Lemma \ref{lem: simplewedge}.

The properties \eqref{item: adm2},\eqref{item: adm3} are immediate.

If $j>i+1$, then clearly $R_{i,j-1,j}\notin A'$.
Also, $T_{i,j-1}\in A'$ since $T_{i,j-1}\in A$.
Finally, $T_{i-1,i}\notin A$ (or $i=1$), since $T$ is a wedge for $A$, so in particular $T_{i-1,i}\notin A'$ (or $i=1$) and the last part follows.
\end{proof}

\begin{lemma}\label{lem: indsmooth_continuation}
Let $\sigma\in S_n$ be such that $\tbl(\sigma)$ is {\adm}, and let $1\leq i<j\leq n$. Then,
\begin{enumerate}
\item $T_{i,j}$ is a wedge for $\tbl(\sigma)$ if and only if $\sigma([i-1])=[i-1]$ and $\sigma(i)\geq j=\sigma^{-1}(i)$.
\label{part: wedget}
\item Suppose that $\sigma\in (S_n)_\smth$ and $T_{i,j}$ is a wedge for $\tbl(\sigma)$. Let $\sigma':=\sigma T_{j-1,j}$. Then,
$$\left(\tbl(\sigma)\right)'=\tbl(\sigma')$$
(where the derived set is taken with respect to $T_{i,j}$).
\end{enumerate}
\end{lemma}

\begin{proof}
Suppose that $T_{i,j}$ is a wedge for $\tbl(\sigma)$. Since $T_{i-1,i}\not\leq\sigma$ (or $i=1$), $\sigma([i-1])=[i-1]$, by \eqref{eq:wedgeT}.
Therefore, since $T_{i,j}\leq\sigma$, we have $\sigma(i)\geq j$ and $\sigma^{-1}(i)\geq j$.
If $j<n$, then since $R_{i,j,j+1}\not\leq\sigma$ but $T_{i,j}\leq\sigma$ we must have
$\maxperm{\sigma^{-1}}(i)=j$, by \eqref{eq:wedgeR}.
Hence, $\sigma^{-1}(i)=j$.

Conversely, if $\sigma([i-1])=[i-1]$ and $\sigma(i)\geq j=\sigma^{-1}(i)$, then it is easy to see
that $T_{i,j}$ is a wedge for $\tbl(\sigma)$.

Assume now that $\sigma$ is smooth and that $T_{i,j}$ is a wedge for $\tbl(\sigma)$.
By the first part, $\sigma([i-1])=[i-1]$ and $\sigma(i)\geq j=\sigma^{-1}(i)$.
Therefore, by Lemma \ref{lem: indsmooth},
\[
\left(\tbl(\sigma')\right)_\Trans=\tbl_\Trans(\sigma')=\tbl_\Trans(\sigma)\setminus\{T_{i,j}\}=\tbl(\sigma)'_\Trans
\]
and $\sigma'$ is smooth, and hence $\tbl(\sigma')$ is {\adm}.
The set $\tbl(\sigma)'$ is {\adm} as well, by the third part of Lemma \ref{lem: simplederived}, and it is easy to see that
$\tbl(\sigma')\subseteq \tbl(\sigma)'$. Therefore, $\tbl(\sigma')=\tbl(\sigma)'$, by Observation \ref{obs:adm_subset}.
\end{proof}

\subsection{}
The last two parts of Lemma \ref{lem: simplederived} justify the following definition.

\begin{definition} \label{def: iterated derived set}
Suppose that $A\subseteq\Special$ is {\adm} and $T_{i,j}$ is a wedge for $A$.
The iterated derived set $A^\circ$ of $A$ with respect to $T_{i,j}$ is the set obtained from $A$ by deriving it repeatedly
$j-i$ times with respect to $T_{i,j}, T_{i,j-1},\ldots,T_{i,i+1}$.
Explicitly,
\begin{align*}
A^\circ=&A\setminus\big(\{T_{i,k}:k>i\}\cup\{L_{i,k,l},R_{i,k,l}:l>k>i\}\big)\\
=&A\setminus\big(\{T_{i,k}:i<k\leq j\}\cup\{L_{i,k,l}:i<k<l\}\cup\{R_{i,k,l}:i<k<l\leq j\}\big).
\end{align*}
In particular,
$$
A^\circ_{\Trans}=A_\Trans\setminus\{T_{i,k}:k>i\}=A_\Trans\setminus\{T_{i,k}:i<k\leq j\}.
$$
\end{definition}

\begin{lemma} \label{lem: simpleiteratedderived}
Suppose that $A\subseteq\Special$ is {\adm} and $T_{i,j}$ is a wedge for $A$.
Then,
\begin{enumerate}
\item $A^\circ$ is {\adm}.
\item If $j>i+1$, then there exists $k\geq j$ such that $T_{i+1,k}$ is a wedge for $A^\circ$ or $(A^\circ)^{-1}$ (or both). \label{part: i+1}
\end{enumerate}
\end{lemma}

\begin{proof}
The first part follows by repeatedly using the last two parts of Lemma \ref{lem: simplederived}.
Suppose that $j>i+1$.
Clearly, $T_{i+1,j}\in A^\circ$ since $T_{i+1,j}\in A$.
Take the maximal $k\geq j$ for which  $T_{i+1,k}\in A^\circ$. Note that $T_{i,i+1}\notin A^\circ$.
Therefore, if $k=n$, then $T_{i+1,k}$ is a wedge for both $A^\circ$ and $(A^\circ)^{-1}$.
If $k<n$, then $T_{i+1,k+1}\notin A^\circ$, by the maximality of $k$, and therefore $R_{i+1,k,k+1}\notin  A^\circ$
or $L_{i+1,k,k+1}\notin  A^\circ$, by the admissability of $A^\circ$ and \eqref{item: adm2}.
If $R_{i+1,k,k+1}\notin  A^\circ$, then $T_{i+1,k}$ is a wedge for $A^\circ$ and if $L_{i+1,k,k+1}\notin  A^\circ$,
i.e., $R_{i+1,k,k+1}\notin  (A^\circ)^{-1}$, then $T_{i+1,k}$ is a wedge for $(A^\circ)^{-1}$.
\end{proof}

\section{Compatible orders}\label{sec:orders}
In this section we define the notion of a compatible order for an {\adm} set.
We show that a compatible order always exists and any two are obtained from one another
by a sequence of elementary operations.

\subsection{}
\begin{definition}
Given an {\adm} subset $A\subseteq\Special$, a compatible order for $A$ is a (strict) total order $\prec$ on $A_\Trans$
such that for all $1\leq i<j<k\leq n$, the following three conditions are satisfied.
\begin{subequations}
\begin{align}
\label{item: compat1} \text{If }& R_{i,j,k}\in A \text{ but } L_{i,j,k}\notin A,\text{ then } T_{i,j}\prec T_{j,k}.\\
\label{item: compat2} \text{If  }& L_{i,j,k}\in A \text{ but } R_{i,j,k}\notin A,\text { then } T_{i,j}\succ T_{j,k}.\\
\label{item: compat3} \text{If }& T_{i,k}\in A,\text{ then either } T_{i,j}\prec T_{i,k}\prec T_{j,k} \text{ or } T_{i,j}\succ T_{i,k}\succ T_{j,k}.
\end{align}
\end{subequations}
\end{definition}

This notion is closely related to \emph{reflection order} (cf. \cite{MR1248893}, \cite[\S5.2]{MR2133266})
except that we do not consider a total order on the whole of $\Trans$.

\begin{remark} \label{rem: firsttwo}
Note that (by the {\admis} of $A$) we can rephrase \eqref{item: compat1}--\eqref{item: compat2} by requiring that for all
$1\leq i<j<k\leq n$ such that $T_{i,j},T_{j,k}\in A_\Trans$ but $T_{i,k}\notin A_\Trans$ we have $T_{i,j}\prec T_{j,k}$
if and only if $R_{i,j,k}\in A$ (or equivalently, if and only if $L_{i,j,k}\notin A$).
\end{remark}

\begin{observation}\label{obs: reverse order}
If $\prec$ is a compatible order for $A$, then the reverse order is a compatible order for $A^{-1}$.
Similarly, $T\prec' T'\iff w_0Tw_0\succ w_0T'w_0$ is a compatible order for $w_0Aw_0$.
\end{observation}

\begin{remark}
Consider $\sigma=w_0$. Then, it is well known that the compatible orders for $\Special=\tbl(w_0)$
are in one-to-one correspondence with the reduced decomposition of $w_0$ \cite{MR911771}.
Their number is given by a well-known formula of Stanley \cite{MR782057}.
In general, the number of reduced decompositions of $\sigma\in(S_n)_{\smth}$ can be either
bigger or smaller than the number of compatible orders for $\tbl(\sigma)$.
\end{remark}

\subsection{}
The following lemma is clear from the definition of $A^\circ$ and the first part of Lemma \ref{lem: simpleiteratedderived}.

\begin{lemma} \label{lem: compatible order+iterated derived set}
Suppose that $\emptyset\ne A\subseteq\Special$ is {\adm} and $T_{i,j}$ is a wedge for $A$. Then,
\begin{enumerate}
\item Any compatible order for $A$ induces a compatible order for $A^\circ$.
\item Any compatible order $\prec^\circ$ for $A^\circ$ may be extended to a compatible order $\prec$ on $A$ by requiring that
\begin{equation} \label{eq: rightalinged}
T_{k,l}\prec T_{i,j}\prec T_{i,j-1}\prec\cdots\prec T_{i,i+1}\text{ for every }T_{k,l}\in A^\circ_\Trans
\end{equation}
\end{enumerate}
\end{lemma}

\begin{corollary} \label{cor: existance of compatible order}
For every {\adm} subset $A\subseteq\Special$ there is a compatible order.
\end{corollary}
\begin{proof}
The corollary follows by induction from the second part of Lemma \ref{lem: compatible order+iterated derived set} using Remark \ref{rem: existwedge} and Observation \ref{obs: reverse order}.
\end{proof}

In general, given a total order $\prec$ on a set, we will write $x\adj y$
if $y$ covers $x$, i.e., if $x\prec y$ and there is no $z$ such that $x\prec z\prec y$.

\begin{lemma}\label{lem:cover}
Let $\emptyset\ne A\subseteq\Special$ be an {\adm} set and let $\prec$ be a compatible order for $A$.
Suppose that $r_1<r_2<s_2$ and $s_1>r_1$ are such that $T_{r_1,s_1}, T_{r_1,s_2}\in A$
and $T_{r_1,s_1}\adj T_{r_1,s_2}\adj T_{r_2,s_2}$. Then $s_1=r_2$.
\end{lemma}

\begin{proof}
Assume on the contrary that $s_1\ne r_2$.
By \eqref{item: compat3}, $T_{r_1,r_2}\prec T_{r_1,s_2}$ and therefore, since $T_{r_1,s_1}\adj T_{r_1,s_2}$,
\begin{equation}\label{eq:r1r2r1s1}
T_{r_1,r_2}\prec T_{r_1,s_1}.
\end{equation}

We first show that $s_1<s_2$. Otherwise, $r_1<r_2<s_2<s_1$.
Then $T_{s_2,s_1}\prec T_{r_1,s_1}$ by \eqref{item: compat3}, and hence $T_{s_2,s_1}\prec T_{r_2,s_2}$.
Therefore, $T_{r_2,s_1}\prec T_{r_2,s_2}$ by \eqref{item: compat3}.
In addition, $T_{r_1,s_1}\prec T_{r_2,s_1}$ by \eqref{item: compat3} and \eqref{eq:r1r2r1s1}.
Hence, we obtained that  $T_{r_1,s_1}\prec T_{r_2,s_1}\prec T_{r_2,s_2}$, in contradiction with $T_{r_1,s_1}\adj T_{r_1,s_2}\adj T_{r_2,s_2}$.

Thus, $s_1<s_2$. Therefore, $T_{r_1,s_2}\prec T_{s_1,s_2}$ by \eqref{item: compat3}, and hence, since $T_{r_1,s_2}\adj T_{r_2,s_2}$,
\begin{equation}\label{eq:r2s2s1s2}
T_{r_2,s_2}\prec T_{s_1,s_2}.
\end{equation}

Assume that $r_1<s_1<r_2<s_2$.
Then $T_{s_1,r_2}\prec T_{r_1,s_1}$ by \eqref{item: compat3} and  \eqref{eq:r1r2r1s1}.
On the other hand, $T_{r_2,s_2}\prec T_{s_1,r_2}$ by \eqref{item: compat3} and \eqref{eq:r2s2s1s2}.
Hence, we obtained that  $T_{r_2,s_2}\prec T_{s_1,r_2}\prec T_{r_1,s_1}$, in contradiction with $T_{r_1,s_1}\prec T_{r_2,s_2}$.

Finally, assume that $r_1<r_2<s_1<s_2$.
Then $T_{r_1,s_1}\prec T_{r_2,s_1}$ by \eqref{item: compat3} and \eqref{eq:r1r2r1s1}.
Additionally, $T_{r_2,s_1}\prec T_{r_2,s_2}$ by \eqref{item: compat3} and \eqref{eq:r2s2s1s2}.
Hence, we obtained that  $T_{r_1,s_1}\prec T_{r_2,s_1}\prec T_{r_2,s_2}$, in contradiction with $T_{r_1,s_1}\adj T_{r_1,s_2}\adj T_{r_2,s_2}$.

Thus, $s_1=r_2$ as required.
\end{proof}

\begin{lemma} \label{lem: another compat}
Let $\emptyset\ne A\subseteq\Special$ be an {\adm} set and $T_{i,j}$ a wedge for $A$.
Let $\prec$ be a compatible order for $A$.
Then, we cannot have $T_{i,j_1}\adj T_{j_1,j_2}$ for any $i<j_1<j_2$.
\end{lemma}

\begin{proof}
Suppose that $T_{i,j_1},T_{j_1,j_2}\in A$ for $i<j_1<j_2$. Then necessarily $j_1\leq j$.
If $j_2>j$ then $R_{i,j_1,j_2}\notin A$ by the second part of Lemma \ref{lem: simplewedge}, and hence $T_{i,j_1}\nprec T_{j_1,j_2}$ by \eqref{item: compat3}.
If $j_2\leq j$ then $T_{i,j_2}\in A$ and hence $T_{i,j_1}\prec T_{i,j_2}\prec T_{j_1,j_2}$ by \eqref{item: compat3}.
\end{proof}

\subsection{}
Suppose that $T_{i,j}\adj T_{k,l}$ and $\{i,j\}\cap\{k,l\}=\emptyset$.
Then, upon switching the order of $T_{i,j}$ and $T_{k,l}$ (but no other elements)
we get a new compatible order. Similarly, if $T_{i,j}\adj T_{i,k}\adj T_{j,k}$ or $T_{j,k}\adj T_{i,k}\adj T_{i,j}$,
then we get a new compatible order by reversing the order of $T_{i,j},T_{i,k},T_{j,k}$ (and otherwise keeping $\prec$).
We call these two operations on compatible orders \emph{elementary}.
We say that two compatible orders are equivalent if they can be obtained from
one another by a finite sequence of elementary operations.

\begin{lemma} \label{lem: admind}
If $A\subseteq\Special$ is {\adm}, then all compatible orders for $A$ are equivalent.
\end{lemma}

\begin{proof}
We will argue by induction on $\#A_\Trans$.
The base of the induction (the case $A=\emptyset$) is trivial. Suppose that $A\ne\emptyset$.
By passing to $A^{-1}$ if necessary, we may assume that there is a wedge $T_{i,j}$ for $A$.
By the induction hypothesis and Lemma \ref{lem: compatible order+iterated derived set}, it suffices to show that any compatible order for $A$
is equivalent to one which satisfies \eqref{eq: rightalinged}.

Following Observation \ref{obs:derived_wedges}, denote ${\mathcal W}=\{T_{i,r}\}_{r=i+1}^j=\{T\in A_\Trans: T(i)\neq i\}$.
For every compatible order $\prec$ for $A$ and every $i<r\leq j$, denote
$$w_r({\prec})=\#\{T\in A_{\Trans}\setminus{\mathcal W}:T_{i,r}\prec T\}.$$
Let $\mathcal E$ be an equivalence class of compatible orders for $A$.
Let $\prec$ be an order in $\mathcal E$ for which the sum $w:=\sum_{r=i+1}^j w_r({\prec})$ is minimal.
We claim that $\prec$ satisfies \eqref{eq: rightalinged}. By \eqref{item: compat3} it is enough to show that $w=0$.

Assume on the contrary that $w>0$, i.e.,
$${\mathcal Z}:=\{T\in A_{\Trans}\setminus{\mathcal W}:T_{i,r}\prec T \text{ for some }i<r\leq j\}\neq\emptyset.$$
Let $T_{k,l}$ be the minimum of $\mathcal Z$ with respect to $\prec$.
Then there is $i<r_1\leq j$ such that $T_{i,r_1}\adj T_{k,l}$.
Note that $k\neq i$ since $T_{k,l}\notin{\mathcal W}$, $l\neq i$ since $T_{i,j}$ is a wedge for $A$,
and $r_1\neq k$ by Lemma \ref{lem: another compat}.
On the other hand, $\{i,r_1\}\cap\{k,l\}\neq\emptyset$, otherwise we could switch the order of $T_{i,r_1},T_{k,l}$
and reduce $w$ by $1$, contradicting the choice of $\prec$.
Therefore $r_1=l$. in particular, $i<l$ and hence $i<k$ since $T_{i,j}$ is a wedge for $A$ and $i\neq k$.
Then, by \eqref{item: compat3}, $T_{i,k}\prec T_{i,l}\adj T_{k,l}$.
By the minimality of $T_{k,l}$ in $\mathcal Z$, it follows that there is $i<r_0\leq j$ such that $T_{i,r_0}\adj T_{i,l}\adj T_{k,l}$.
By Lemma \ref{lem:cover}, $r_0=k$. Then we could switch the order of $T_{i,k},T_{k,l}$ and reduce $w$ by $2$, contradicting the choice of $\prec$.
\end{proof}

Lemma \ref{lem: admind} and the braid relations
\[
T_{i,j}T_{i,k}T_{j,k}=T_{j,k}T_{i,k}T_{i,j},\ i<j<k,
\]
immediately imply the following corollary.
\begin{corollary}\label{cor:def_pi}
Let $A\subseteq\Special$ be {\adm} and let $\prec$ be a compatible order for $A$.
Write $A_\Trans=\{\sigma_1,\ldots,\sigma_k\}$ with $\sigma_1\prec\cdots\prec\sigma_k$.
Then, the product $\pi(A):=\sigma_1\cdots\sigma_k\in S_n$ depends only on $A$ and not on the choice of $\prec$.
\end{corollary}

\begin{observation}\label{obs:aut_pi}
For every {\adm} $A\subseteq\Special$ we have, in light of Observation \ref{obs: reverse order},
\begin{equation*}
\pi(A^{-1})=(\pi(A))^{-1},\quad \pi(w_0Aw_0)=w_0\pi(A)^{-1}w_0.
\end{equation*}
\end{observation}

\begin{remark}
Suppose that $\prec$ is a total order on $A_\Trans$ and write $A_\Trans=\{\sigma_1,\ldots,\sigma_k\}$
with $\sigma_1\prec\cdots\prec\sigma_k$.
It is possible that the product $\sigma_1\cdots\sigma_k$ is equal to $\pi(A)$ even if
$\prec$ is not compatible with respect to $A$.
For instance, if $A=\Special_4$, then there are 64 total orders on $\Trans$ with this property
(i.e., 64 ways to write $w_0$ as the product of all transpositions)
and only 16 of them are compatible with respect to $\Special$.
\end{remark}

\section{The main bijection}\label{sec:bijection}

In this section we prove Theorem \ref{thm:bijection}.

\subsection{}

Recall that $\rshft ij\in S_n$, $i<j$ is the cycle permutation $i\to i+1\to\cdots\to j\to i$

\begin{lemma} \label{lem:pi}
Suppose that $A\subseteq\Special$ is {\adm} and $T_{i,j}$ is a wedge for $A$.
Let $\sigma=\pi(A)$. Then,
\begin{enumerate}
\item $\sigma=\pi(A^\circ)\rshft ij$.
\item $\sigma=\pi(A')T_{j-1,j}$.
\item $\sigma(j)=i$. \label{part: sji}
\item $\sigma([i-1])=[i-1]$. \label{part: si-1}
\item $\sigma(k)\geq i+j-k$ and $\sigma^{-1}(k)\geq i+j-k$ for every $i\leq k\leq j$. \label{part: i+j-k}
\item $\sigma(i)=j\implies\sigma(k)=i+j-k$ for all $i< k\leq j$. \label{part: sigmai=j}
\end{enumerate}
\end{lemma}

\begin{proof}
For the first part, by considering a compatible order for $A$ satisfying \eqref{eq: rightalinged}, we have
$$\sigma=\pi(A)=\pi(A^\circ)T_{i,j}T_{i,j-1}\cdots T_{i,i+1}=\pi(A^\circ)\rshft ij.$$
The second part follows from the first part if $j=i+1$ (in which case $A^\circ=A'$ and $\rshft ij=T_{j-1,j}$).
On the other hand, if $j>i+1$, then $\sigma$ is equal to
\[
\pi(A^\circ)\rshft ij=\pi(A^\circ)\rshft i{j-1}T_{j-1,j}=\pi\left((A')^{\circ}\right)\rshft i{j-1}T_{j-1,j}=\pi(A')T_{j-1,j}
\]
by applying the first part to $A'$ and using the last part of Lemma \ref{lem: simplederived}.

Moreover, since $A^{\circ}$ does not contain any transposition of the form $T_{r,i}$ or $T_{i,l}$, we have
$$\sigma(j)=\pi(A^\circ)\left(\rshft ij(j)\right)=\pi(A^\circ)(i)=i,$$
proving part three.

Since $T_{r,s}\notin A$ if $r<i\leq s$, we have $\sigma([i-1])=[i-1]$, i.e., the fourth part.

For the last two parts we use induction on the size of $A$.
If $j=i+1$, then the claims of part \ref{part: i+j-k} and part \ref{part: sigmai=j} follow directly from
parts \ref{part: sji} and \ref{part: si-1}.
Therefore, assume that $j>i+1$. By Lemma \ref{lem: simpleiteratedderived} part \ref{part: i+1} there exists $r\geq j$ such that
$T_{i+1,r}$ is a wedge for $A^\circ$ or $(A^\circ)^{-1}$. Since $\pi((A^\circ)^{-1})=(\pi(A^\circ))^{-1}$, we have for every $i<k\leq r$,
by the induction hypothesis,
\[
(\pi(A^\circ))(k),(\pi(A^\circ))^{-1}(k)\geq i+1+r-k,
\]
and if $(\pi(A^\circ))(i+1)=r$ then $(\pi(A^\circ))(k)=i+1+r-k$ for every $i+1<k\leq r$.
Therefore, by the first part, for every $i\leq k<j$,
\begin{equation}\label{eq:}\sigma(k)=\pi(A^\circ)(\rshft ij(k))=(\pi(A^\circ))(k+1)\geq i+1+r-(k+1)\geq i+j-k,\end{equation}
$\sigma(j)=i=i+j-j$, and for every $i\leq k\leq j$
$$\sigma^{-1}(k)=\rshft ij^{-1}(\pi(A^\circ)^{-1}(k))\geq (\pi(A^\circ))^{-1}(k)-1\geq i+r-k\geq i+j-k.$$
Moreover, if $\sigma(i)=j$ then by \eqref{eq:}, $(\pi(A^\circ))(i+1)=r=j$ and hence
\begin{equation*}
\sigma(k)=(\pi(A^\circ))(k+1)= i+1+r-(k+1)= i+j-k
\end{equation*}
 for every $i < k\leq r-1=j-1$ (and obviously $\sigma(j)=i=i+j-j$).
\end{proof}

Recall that $\tbl(\sigma)$ is {\adm} for every smooth $\sigma$, by Lemma \ref{lem: covexisadm}.
\begin{proposition}\label{propos:order}
For every $\sigma\in(S_n)_\smth$ it holds that $\pi(\tbl(\sigma))=\sigma$.
\end{proposition}

\begin{proof}
We argue by induction on $\ell(\sigma)$.
The base of the induction (the case where $\sigma$ is the identity permutation) is trivial.
For the induction step, let $i$ be the smallest index such that $j:=\sigma^{-1}(i)\ne i$. Passing to $\sigma^{-1}$ if necessary
we may assume that $\sigma(i)\geq j$. Then, $\sigma'=\sigma T_{j-1,j}$ is smooth by Lemma \ref{lem: indsmooth}.
By Lemma \ref{lem: indsmooth_continuation}, the transposition $T_{i,j}$ is a wedge for $\tbl(\sigma)$ and the derived set is $\tbl(\sigma')$.
Hence, by the second part of Lemma \ref{lem:pi} and the  induction hypothesis,
\begin{equation*}
\pi(\tbl(\sigma))=\pi\left((\tbl(\sigma))'\right)T_{j-1,j}=\pi(\tbl(\sigma'))T_{j-1,j}=\sigma'T_{j-1,j}=\sigma.
\qedhere\end{equation*}
\end{proof}

\begin{proposition}\label{propos:bijection}
For every {\adm} $A\subseteq\Special$,
the permutation $\pi(A)$ is smooth and $\tbl(\pi(A))=A$
\end{proposition}

\begin{proof}
We argue by induction on the size of $A$.
The base of the induction (the case $A=\emptyset$) is trivial.
For the induction step, passing to $A^{-1}$ if necessary, we may assume that there is a wedge $T_{i,j}$ for $A$.
For simplicity, denote $\sigma:=\pi(A)$ and $\sigma':=\sigma T_{j-1,j}$.
By Lemma \ref{lem:pi}, $\sigma'=\pi(A')$, $\sigma([i-1])=[i-1]$, $\sigma(i)\geq j=\sigma^{-1}(i)$ and $\sigma^{-1}(i+1)\geq j-1$.
By the induction hypothesis, the permutation $\sigma'=\pi(A')$ is smooth and $\tbl(\sigma')=A'$.

We show that $\tbl_\Trans(\sigma)=\tbl_\Trans(\sigma')\cup\{T_{i,j}\}$.
It is clear that $\tbl_\Trans(\sigma')\subseteq \tbl_\Trans(\sigma)$ and that $T_{i,j}\in\tbl_\Trans(\sigma)$.
Hence, $\tbl_\Trans(\sigma')\cup\{T_{i,j}\}\subseteq\tbl_\Trans(\sigma)$.
Conversely, suppose that $T_{r,s}\in\tbl_\Trans(\sigma)$.
By \eqref{eq: tblll}, if $r\ne j-1$ and $s\ne j$, then $T_{r,s}\in\tbl_\Trans(\sigma')$.
If $s=j$, then either $r=i$ or $T_{r,s}\in\tbl_\Trans(\sigma')$ since $\sigma^{-1}(i+1)\geq j-1$.
Suppose now that $r=j-1$ and $s\neq j$.
Then, $r>i$ and $T_{j,s}\in\tbl_\Trans(\sigma)$ and hence by \eqref{eq: tblll},
$T_{j,s}\in\tbl_\Trans(\sigma')$.
By Lemma \ref{lem: simplederived} part \ref{item: LkjlinA'}, $L_{r,j,s}\in A'=\tbl(\sigma')$.
In particular, $\maxperm{\sigma'}(r)\geq s$.
In light of  Observation \ref{obs:sub_nsub}, the condition $\maxperm{\sigma'^{-1}}(r)\geq s$ also holds,
since $\maxperm{\sigma^{-1}}(r)\geq s$ and $s\neq j$.
Hence $T_{r,s}\in\tbl_\Trans(\sigma')$. In conclusion, $\tbl_\Trans(\sigma)=\tbl_\Trans(\sigma')\cup\{T_{i,j}\}$ as claimed.

It follows from Lemma \ref{lem: indsmooth} that $\sigma$ is smooth.
In particular, $\tbl(\sigma)$ is {\adm}.
Finally, by Lemma \ref{lem: indsmooth_continuation}, $T_{i,j}$ is a wedge for $\tbl(\sigma)$ and $\tbl(\sigma)'=\tbl(\sigma')=A'$.
Hence, $\tbl(\sigma)=A$, by Lemma \ref{lem: simplederived} part \ref{item:recover}.
The proposition follows.
\end{proof}

Note that Proposition \ref{propos:order} and Proposition \ref{propos:bijection} do not yet finish the proof of
Theorem \ref{thm:bijection} since we still have to show the relation \eqref{eq: pialt}.

\subsection{}\label{sec:cycles}

For every non-empty subset $A=\{i_1,\ldots,i_k\}\subseteq [n]$ with $i_1<\cdots<i_k$ let $\cycr_A$ be the cycle permutation
$i_1\to i_2\to\cdots\to i_k\to i_1$ and let $\cycl_A:=\cycr_A^{-1}$.
Note that this is consistent with the notation $\cycr_{[i,j]}$ introduced before.
Denote by $\Cycles=\Cycles_n$ the set of permutations of the form $\cycr_A$ or $\cycl_A$
for some $\emptyset\neq A\subseteq [n]$, and let
\[
\Cycles_{\spcl}=(\Cycles_n)_{\spcl}:=\{\rshft ij\}_{1\leq i<j\leq n}\cup\{\lshft ij\}_{1\leq i<j\leq n}.
\]
Note that $\cycr_A=\cycl_A=T_{i,j}$ if $A=\{i,j\}$ and $\cycr_A=\cycl_A=e$ if $A$ is a singleton.
Thus,
\[
\#\Cycles_n = 2^{n+1}-\binom{n}2-2n-1,
\]
whereas
\begin{equation*}
\#\Trans_n=\binom n2,\ \#(\Cycles_n)_{\spcl}=2\binom n2-(n-1)=(n-1)^2,\ \#\Special_n=2\binom n3+\binom n2.
\end{equation*}
It is easy to see that for every $\sigma\in S_n$ and $\emptyset\ne A=\{i_1,\ldots,i_k\}\subseteq[n]$ with $i_1<\cdots<i_k$ we have
\begin{subequations} \label{eq:cyc1<=}
\begin{align}
\label{eq:cycr1<=}\cycr_A\leq\sigma&\iff\maxperm{\sigma}(i_j)\geq i_{j+1}\text{ for all }1\leq j<k \text{ and }\maxperm{\sigma^{-1}}(i_1)\geq i_k,\\
\label{eq:cycl1<=}\cycl_A\leq\sigma&\iff\maxperm{\sigma^{-1}}(i_j)\geq i_{j+1}\text{ for all }1\leq j<k \text{ and }\maxperm{\sigma}(i_1)\geq i_k.
\end{align}

In particular, for every $i<j$
\begin{equation} \label{eq: cycrs}
\rshft ij\leq\sigma\iff\maxperm{\sigma^{-1}}(i)\geq j.
\end{equation}
Indeed, if $\maxperm{\sigma^{-1}}(i)\geq j$, then for all $r\in[i,j-1]$,
$\maxperm{\sigma}(r)\geq r+1$, otherwise $\sigma([r])=[r]$ and in particular,
$\maxperm{\sigma^{-1}}(i)\leq\maxperm{\sigma^{-1}}(r)=r<j$.
Similarly,
\begin{equation} \label{eq: cycls}
\lshft ij\leq\sigma\iff\maxperm{\sigma}(i)\geq j.
\end{equation}

\begin{observation}\label{obs:order_cycles1} ~
\begin{enumerate}
\item Let $A\subseteq[n]$ be a set consisting of at least two elements. Then,
\begin{equation*}
\cycl_{A\setminus\{\min A\}}\leq\cycl_A,\quad \cycr_{A\setminus\{\min A\}}\leq\cycr_A,\quad \cycl_{A\setminus\{\max A\}}\leq\cycl_A, \quad \cycr_{A\setminus\{\max A\}}\leq\cycr_A.
\end{equation*}
\item Let $\emptyset\neq A\subset B\subseteq [n]$ be sets such that $\min A=\min B$ and $\max A=\max B$. Then,
\begin{equation*}
\cycl_B\leq\cycl_A,\quad \cycr_B\leq\cycr_A.
\end{equation*}
\end{enumerate}
\end{observation}

\end{subequations}
For every $\sigma\in S_n$ let
\[
\tblc(\sigma)=\{\tau\in\Cycles:\tau\leq\sigma\},
\]
\[
\tblc_{\spcl}(\sigma)=\{\tau\in\Cycles_{\spcl}:\tau\leq\sigma\}=
\tblc(\sigma)\cap\Cycles_{\spcl}.
\]

We say that $\sigma\in S_n$ is {\dbi} if for every $\tau\in S_n$ we have
\[
\tau\leq\sigma\impliedby\maxperm{\tau}(i)\leq\maxperm{\sigma}(i)
\text{ and }\maxperm{\tau^{-1}}(i)\leq\maxperm{\sigma^{-1}}(i)\text{ for all }i.
\]
(The implication $\implies$ always holds, cf. \eqref{eq:Bruhat}.)
This terminology conforms with a similar notion for Schubert varieties \cite{MR1934291}.

\begin{lemma}\label{lem:tblc_dbi}
The following three conditions are equivalent for $\sigma\in S_n$.
\begin{enumerate}
\item \label{part: definc} $\sigma$ is {\dbi}.
\item \label{part: tblc} For every $\tau\in S_n$, $\tau\leq\sigma$ if and only if $\tblc(\tau)\subseteq\tblc(\sigma)$.
\item \label{part: tblcspcl} For every $\tau\in S_n$, $\tau\leq\sigma$ if and only if $\tblc_{\spcl}(\tau)\subseteq\tblc_{\spcl}(\sigma)$.
\end{enumerate}
\end{lemma}

\begin{proof}
Clearly, if $\tau\leq\sigma$, then $\tblc(\tau)\subseteq\tblc(\sigma)$ and $\tblc_{\spcl}(\tau)\subseteq\tblc_{\spcl}(\sigma)$.
Also, if $\tblc(\tau)\subseteq\tblc(\sigma)$, then $\tblc_{\spcl}(\tau)\subseteq\tblc_{\spcl}(\sigma)$.
Thus, \ref{part: tblcspcl}$\implies$\ref{part: tblc}.

Conversely, if $\tblc_{\spcl}(\tau)\subseteq\tblc_{\spcl}(\sigma)$, then by \eqref{eq:cyc1<=} we have $\tblc(\tau)\subseteq\tblc(\sigma)$. Hence \ref{part: tblc}$\implies$\ref{part: tblcspcl}.

The equivalence of conditions \ref{part: definc} and \ref{part: tblcspcl} follows from \eqref{eq: cycrs} and \eqref{eq: cycls}.
\end{proof}

\begin{lemma}\label{lem:covex_cyc}
Let $\sigma$ be a covexillary permutation in $S_n$, let $r\geq 3$ and $1\leq i_1<\cdots<i_r\leq n$.

If $T_{i_1,i_{r-1}}\nleq\sigma$, $T_{i_2,i_r}\nleq\sigma$, $\cycl_{i_1,\ldots,i_{r-1}}\leq\sigma$ and $\cycl_{i_2,\ldots,i_r}\leq\sigma$, then $\cycl_{i_1,\ldots,i_r}\leq\sigma$.

Similarly, if $T_{i_1,i_{r-1}}\nleq\sigma$, $T_{i_2,i_r}\nleq\sigma$, $\cycr_{i_1,\ldots,i_{r-1}}\leq\sigma$ and $\cycr_{i_2,\ldots,i_r}\leq\sigma$, then $\cycr_{i_1,\ldots,i_r}\leq\sigma$.
\end{lemma}

\begin{proof}
First note that the statements are empty if $r=3$, so we may assume $r>3$.
We only need to prove the first statement, as the second one would then follow by passing to $\sigma^{-1}$.
Since $\cycl_{i_1,\ldots,i_{r-1}}\leq\sigma$ and $\cycl_{i_2,\ldots,i_r}\leq\sigma$, it follows from \eqref{eq:cycl1<=} that $\maxperm{\sigma}(i_1)\geq i_{r-1}$, $\maxperm{\sigma}(i_2)\geq i_r$ and $\maxperm{\sigma^{-1}}(i_t)\geq i_{t+1}$ for every $1\leq t\leq r-1$.
Since $T_{i_1,i_{r-1}}\nleq\sigma$, it follows that $\maxperm{\sigma^{-1}}(i_1)<i_{r-1}$.
Similarly, $\maxperm{\sigma^{-1}}(i_2)<i_r$, since $T_{i_2,i_r}\nleq\sigma$.

By way of contradiction assume now that $\cycl_{i_1,\ldots,i_r}\nleq\sigma$.
Since $\maxperm{\sigma^{-1}}(i_t)\geq i_{t+1}$ for every $1\leq t\leq r-1$, it follows that $\maxperm{\sigma}(i_1)<i_r$.

Since $i_{r-1}\leq\maxperm{\sigma}(i_1)<i_r$, there is $x\leq i_1$ such that $i_{r-1}\leq\sigma(x)<i_r$.
Similarly, since $i_2\leq\maxperm{\sigma^{-1}}(i_1)<i_{r-1}$, there is $i_2\leq z< i_{r-1}$ such that $\sigma(z)\leq i_1$.
Since $\maxperm{\sigma}(i_1)<i_r\leq\maxperm{\sigma}(i_2)$, there is $i_1<y\leq i_2$ such that $\sigma(y)\geq i_r$.
Finally, since $\maxperm{\sigma^{-1}}(i_2)<i_r\leq\maxperm{\sigma^{-1}}(i_{r-1})$, there is $ w\geq i_r$ such that $i_2<\sigma(w)\leq i_{r-1}$.

Therefore, $x<y\leq z<w$ and $\sigma(z)<\sigma(w)\leq\sigma(x)<\sigma(y)$, violating the covexillarity of $\sigma$.
\end{proof}

The following observation immediately follows from \eqref{eq:cycr1<=} and \eqref{eq:cycl1<=}.
\begin{observation}\label{obs:order_cycles2}
Suppose that $\tau\in S_n$, $r\geq 3$ and $1\leq i_1<\cdots<i_r\leq n$. Then,
\begin{enumerate}
\item $T_{i_1,i_{r-1}}\vee\cycl_{i_1,\ldots,i_r}=L_{i_1,i_{r-1},i_r}$.
\item $T_{i_1,i_{r-1}}\vee\cycr_{i_1,\ldots,i_r}=R_{i_1,i_{r-1},i_r}$.
\item $T_{i_2,i_r}\vee\cycl_{i_1,\ldots,i_r}=L_{i_1,i_2,i_r}$.
\item $T_{i_2,i_r}\vee\cycr_{i_1,\ldots,i_r}=R_{i_1,i_2,i_r}$.
\end{enumerate}
\end{observation}

\begin{corollary} \label{cor:order_cycles}
Let $\sigma$ be a covexillary permutation in $S_n$, $r\geq 3$ and $1\leq i_1<\cdots<i_r\leq n$.
Then, $\cycl_{i_1,\ldots,i_r}\leq\sigma$ if and only if at least one of the following three conditions holds.
\begin{itemize}
\item $L_{i_1,i_{r-1},i_r}\leq\sigma$,
\item $L_{i_1,i_2,i_r}\leq\sigma$,
\item $T_{i_1,i_{r-1}}\nleq\sigma$, $T_{i_2,i_r}\nleq\sigma$, $\cycl_{i_1,\ldots,i_{r-1}}\leq\sigma$ and $\cycl_{i_2,\ldots,i_r}\leq\sigma$.
\end{itemize}
Similarly, $\cycr_{i_1,\ldots,i_r}\leq\sigma$ if and only if at least one of the following three conditions holds.
\begin{itemize}
\item $R_{i_1,i_{r-1},i_r}\leq\sigma$.
\item $R_{i_1,i_2,i_r}\leq\sigma$.
\item $T_{i_1,i_{r-1}}\nleq\sigma$, $T_{i_2,i_r}\nleq\sigma$, $\cycr_{i_1,\ldots,i_{r-1}}\leq\sigma$ and $\cycr_{i_2,\ldots,i_r}\leq\sigma$.
\end{itemize}
\end{corollary}

Indeed, this follows from Observation \ref{obs:order_cycles1}, Observation \ref{obs:order_cycles2} and Lemma \ref{lem:covex_cyc}.

The corollary provides for any covexillary permutation $\sigma$,
a simple recursive algorithm for constructing $\tblc(\sigma)$ out of its subset $\tbl(\sigma)$.

\begin{corollary}\label{cor:tbl_tblc}
Suppose that $\sigma\in S_n$ is covexillary, $\tau\in S_n$ and $\tbl_\Trans(\tau)=\tbl_\Trans(\sigma)$.
Then, $\tblc(\tau)\subseteq\tblc(\sigma)$ if and only if $\tbl(\tau)\subseteq\tbl(\sigma)$.
\end{corollary}

\begin{proof}
If $\tblc(\tau)\subseteq\tblc(\sigma)$, then clearly $\tbl(\tau)=\tblc(\tau)\cap\Special\subseteq\tblc(\sigma)\cap\Special=\tbl(\sigma)$ .

To show the opposite implication, we prove by induction on $r$ that if
$\cycl_{i_1,\ldots,i_r}\leq \tau$ for some $i_1<\cdots<i_r$, then $\cycl_{i_1,\ldots,i_r}\leq \sigma$.
Similarly, by passing to $\sigma^{-1}$, if $\cycr_{i_1,\ldots,i_r}\leq \tau$, then $\cycr_{i_1,\ldots,i_r}\leq \sigma$.
The base cases $r=2$ and $r=3$ are given. For the induction step, let $r\geq 4$
and assume that the induction hypothesis is satisfied for $r-1$.

Suppose first that $T_{i_1,i_{r-1}}\leq\tau$.
Then, $L_{i_1,i_{r-1},i_r}\leq\tau$, by Observation \ref{obs:order_cycles2}.
Therefore, $L_{i_1,i_{r-1},i_r}\in\tbl(\tau)$ and hence $L_{i_1,i_{r-1},i_r}\in\tbl(\sigma)$, i.e., $L_{i_1,i_{r-1},i_r}\leq\sigma$.
Hence, $\cycl_{i_1,\ldots,i_r}\leq \sigma$, by Observation \ref{obs:order_cycles1}.
A similar argument shows that $\cycl_{i_1,\ldots,i_r}\leq \sigma$ if $T_{i_2,i_r}\leq\tau$.

Finally, suppose that $T_{i_1,i_{r-1}}\nleq\tau$ and $T_{i_2,i_r}\nleq\tau$.
Then, $T_{i_1,i_{r-1}},T_{i_2,i_r}\notin\tbl_\Trans(\tau)=\tbl_\Trans(\sigma)$
and hence $T_{i_1,i_{r-1}}\nleq\sigma$ and $T_{i_2,i_r}\nleq\sigma$.
On the other hand, by Observation \ref{obs:order_cycles1}, $\cycl_{i_1,\ldots,i_{r-1}}\leq\tau$ and $\cycl_{i_2,\ldots,i_r}\leq\tau$
and hence, by the induction hypothesis, $\cycl_{i_1,\ldots,i_{r-1}}\leq\sigma$ and $\cycl_{i_2,\ldots,i_r}\leq\sigma$.
It follows  from Lemma \ref{lem:covex_cyc} that $\cycl_{i_1,\ldots,i_r}\leq \sigma$.
\end{proof}

By \cite{MR1934291}, $\sigma$ is {\dbi} if and only if $\sigma$ is $4231$, $35142$, $42513$ and $351624$ avoiding.
(See \cite{MR3644818} and the references therein for other equivalent conditions.)
In particular, a permutation is smooth if and only if it is both covexillary and {\dbi}.

\begin{corollary}\label{cor:smoothorder}
Suppose that $\tau\in S_n$, $\sigma\in (S_n)_\smth$ and $\tbl_\Trans(\tau)=\tbl_\Trans(\sigma)$.
Then, $\tau\leq\sigma$ if and only if $\tbl(\tau)\subseteq\tbl(\sigma)$.
\end{corollary}

\begin{proof}
Clearly, if  $\tau\leq\sigma$, then $\tbl(\tau)\subseteq\tbl(\sigma)$.
Conversely, suppose that $\sigma$ is smooth, $\tbl(\tau)\subseteq\tbl(\sigma)$ and $\tbl_\Trans(\tau)=\tbl_\Trans(\sigma)$.
By Corollary \ref{cor:tbl_tblc}, $\tblc(\tau)\subseteq\tblc(\sigma)$ since $\sigma$ is covexillary.
Hence, by Lemma \ref{lem:tblc_dbi}, $\tau\leq\sigma$, since $\sigma$ is {\dbi}, as required.
\end{proof}

\begin{remark}
It is not true in general that $\tbl(\tau)\subseteq\tbl(\sigma)$ implies that $\tau\leq\sigma$
even if $\sigma,\tau\in (S_n)_\smth$. For instance, if $\tau=\rshft 1n$ and $\sigma=w_0 T_{n-1,n}=T_{1,2}w_0$,
$n>1$, then $\tau\not\leq\sigma$ since $\tau(n)=1<2=\sigma(n)$. On the other hand,
\[
\tbl(\tau)=\{T_{i-1,i}:i\in[n-1]\}\cup\{R_{i-1,i,i+1}:i\in[2,n-1]\}
\]
and
\[
\tbl(\sigma)=\tbl(w_0)'=\Special\setminus\{T_{1,n},R_{1,i,n}:1<i<n\},
\]
so that $\tbl(\tau)\subseteq\tbl(\sigma)$ if $n>3$. Note that among all pairs of permutations in $S_n$ such that $\tau\not\leq\sigma$,
$\ell(\sigma)-\ell(\tau)=\binom n2-n$ is maximal in the example above.
\end{remark}

We can now finish the proof of Theorem \ref{thm:bijection}.

In view of Proposition \ref{propos:order} and Proposition \ref{propos:bijection}
it remains to prove  the relation \eqref{eq: pialt}.
Let $A\subset\Special$ be {\adm} and $\sigma=\pi(A)$.
By Proposition \ref{propos:bijection}, $\sigma$ is smooth and $\tbl(\sigma)=A$.
On the other hand, if $\tau$ is a permutation such that $\tbl_\Trans(\tau)=A_\Trans$ and $\tbl(\tau)\subseteq A$,
then $\tau\leq\sigma$ by Corollary \ref{cor:smoothorder}. \qed

\begin{question}
Given $\sigma_1,\sigma_2\in (S_n)_{\smth}$ with $\sigma_1\leq\sigma_2$, does there
exist a compatible order for $\tbl(\sigma_2)$ whose restriction to $\tbl_{\Trans}(\sigma_1)$ is a compatible order for $\tbl(\sigma_1)$?
\end{question}

\section{An application} \label{sec:misc}

Let $\sim$ be an equivalence relation on $[n]$ and let $\prtt$ be the set of equivalence classes of $\sim$
(i.e., the corresponding partition of $[n]$).
We denote by $S_{\prtt}$ the subgroup of $S_n$ of permutations that preserve every equivalence class.
In this section we prove Theorem \ref{thm: bruintH} (see Proposition \ref{prop: charsm} below).

\subsection{}

For each equivalence class $\prtset\in \prtt$ let $\eta_{\prtset}$ be the order preserving bijection $[\# \prtset]\to \prtset$
and let $\iota_{\prtset}:S_{\# \prtset}\rightarrow S_{\prtt}$ be the injective homomorphism given by
\[
\iota_{\prtset}(\sigma)(\eta_{\prtset}(i))=\eta_{\prtset}(\sigma(i)),\ \ i\in [\# \prtset],\ \ \iota_{\prtset}(\sigma)(i)=i,\ \ \forall i\notin \prtset.
\]
Let $\iota=(\iota_{\prtset})_{\prtset\in \prtt}$ be the isomorphism
\[
\iota:\prod_{\prtset\in \prtt}S_{\# \prtset}\rightarrow S_{\prtt}.
\]

Denote by $\leq_{\prtt}$ the partial order on $S_{\prtt}$ obtained from
the product order on $\prod_{\prtset\in \prtt}S_{\# \prtset}$ via $\iota$.
Note that $\leq_{\prtt}$ is stronger than the Bruhat order on $S_{\prtt}$ (induced from $S_n$).
It is strictly stronger if there exist indices $i<j<k<l$ such that $i\sim l\not\sim j\sim k$.

Also, note that in general $\iota$ does not preserve smoothness, i.e.,
\[
\iota\left(\prod_{\prtset\in \prtt}(S_{\# \prtset})_{\smth}\right)\nsubseteq (S_n)_{\smth}.
\]
(For instance, for $n=4$, the non-smooth permutation $(3412)$ is in the image of
$\iota:S_2\times S_2\rightarrow S_4$ with respect to the equivalence relation $i\sim j\iff i\equiv j\pmod2$.)

For any $\prtset\in \prtt$ and $A\subseteq S_n$ let $A_{\prtset}=\iota_{\prtset}^{-1}(A)\subset S_{\# \prtset}$.

\begin{observation}\label{obs:iota_adm}
For every {\adm} $A\subseteq\Special_n$ and $\prtset\in \prtt$, the set $A_{\prtset}\subseteq\Special_{\# \prtset}$ is {\adm}.
\end{observation}

\begin{lemma}\label{lem:tblc_r}
Let $\sigma\in S_n$, $\prtset\in \prtt$ and $\sigma'\in S_{\# \prtset}$.
Assume that $\sigma$ and $\sigma'$ are covexillary and $\tbl(\sigma')=(\tbl(\sigma))_{\prtset}$.
Then, $\tblc(\sigma')=(\tblc(\sigma))_{\prtset}$.
\end{lemma}

\begin{proof}
For every $\tau\in \Special_{\# \prtset}$, $\iota_{\prtset}(\tau)\leq\sigma$, i.e., $\iota_{\prtset}(\tau)\in\tbl(\sigma)$ if and only if
$\tau\in (\tbl(\sigma))_{\prtset}=\tbl(\sigma')$, i.e., $\tau\leq\sigma'$.
Using Corollary \ref{cor:order_cycles}, it now follows by induction on $\#I$ that for every $\emptyset\neq I\subseteq[\# \prtset]$,
\begin{align*}
\iota_{\prtset}(\cycl_I)\leq\sigma &\iff \cycl_I\leq\sigma',\\
\iota_{\prtset}(\cycr_I)\leq\sigma &\iff \cycr_I\leq\sigma',
\end{align*}
and the lemma follows.
\end{proof}

\begin{lemma}\label{lem:1perm}
For every $\sigma\in (S_n)_\smth$ and $\prtset\in \prtt$ we have
\[
\sigma[\prtset]:=\max\{\tau\in S_{\# \prtset}:\iota_{\prtset}(\tau)\leq\sigma\}\in (S_{\# \prtset})_\smth.
\]
Moreover, $\sigma[\prtset]=e$ (i.e.,
\[
\{\tau\in S_{\# \prtset}:\iota_{\prtset}(\tau)\leq\sigma\}=\{e\}),
\]
if and only if there do not exist $i<j$ in $\prtset$ such that $T_{i,j}\leq\sigma$.
\end{lemma}

\begin{proof}
The set $A:=\tbl(\sigma)$ is an {\adm} subset of $\Special_n$ by Lemma \ref{lem: covexisadm}.
Hence, the set $A_{\prtset}\subseteq\Special_{\# \prtset}$ is {\adm}, by Observation \ref{obs:iota_adm}.
Define provisionally $\pi_{\prtset}=\pi(A_{\prtset})\in S_{\# \prtset}$.
By Proposition \ref{propos:bijection}, $\pi_{\prtset}$ is smooth and $\tbl(\pi_{\prtset})=A_{\prtset}$.
We need to show that $\pi_{\prtset}=\max\{\tau\in S_{\# \prtset}:\iota_{\prtset}(\tau)\leq\sigma\}$.

We first show that $\iota_{\prtset}(\pi_{\prtset})\leq\sigma$.
For every $i\in[n]$, let $i_{\prtset}:=\#([i]\cap \prtset)$ and $j_{\prtset}=\maxperm{\pi_{\prtset}}(i_{\prtset})$.
By \eqref{eq: cycls}, $\lshft {i_{\prtset}}{j_{\prtset}}\leq\pi_{\prtset}$,
i.e., $\lshft {i_{\prtset}}{j_{\prtset}}\in\tblc(\pi_{\prtset})$.
Hence, by Lemma \ref{lem:tblc_r},
$\lshft {i_{\prtset}}{j_{\prtset}}\in(\tblc(\sigma))_{\prtset}$, that is
$$\cycl_{\eta_{\prtset}([i_{\prtset},j_{\prtset}])}=
\iota_{\prtset}(\lshft {i_{\prtset}}{j_{\prtset}})\in \tblc(\sigma),$$
i.e., $\cycl_{\eta_{\prtset}([i_{\prtset},j_{\prtset}])}\leq\sigma$.
In particular, by \eqref{eq:cycl1<=}, $\maxperm{\sigma}(\eta_{\prtset}(i_{\prtset}))\geq \eta_{\prtset}(j_{\prtset})$ and hence
$\eta_{\prtset}(j_{\prtset})\leq\maxperm{\sigma}(i)$, since $\eta_{\prtset}(i_{\prtset})\leq i$.
Therefore,
$$\maxperm{\iota_{\prtset}(\pi_{\prtset})}(i)=\max\{\eta_{\prtset}(j_{\prtset}),\,i\}\leq\maxperm{\sigma}(i).$$
Similarly, $\maxperm{(\iota_{\prtset}(\pi_{\prtset}))^{-1}}(i)=\maxperm{\iota_{\prtset}(\pi_{\prtset}^{-1})}(i)\leq\maxperm{\sigma^{-1}}(i)$
and hence $\iota_{\prtset}(\pi_{\prtset})\leq\sigma$, since $\sigma$ is {\dbi}.

Conversely, let $\tau\in S_{\# \prtset}$ be such that $\iota_{\prtset}(\tau)\leq \sigma$.
Then, $\tblc(\iota_{\prtset}(\tau))\subseteq\tblc(\sigma)$ and hence $\left(\tblc(\iota_{\prtset}(\tau))\right)_{\prtset}\subseteq(\tblc(\sigma))_{\prtset}$.
Clearly,
\begin{equation*}
\left(\tblc(\iota_{\prtset}(\tau))\right)_{\prtset}=\{\rho\in S_{\# \prtset}:\iota_{\prtset}(\rho)\in\Cycles_n,\,\iota_{\prtset}(\rho)\leq\iota_{\prtset}(\tau)\}=
\{\rho\in \Cycles_{\# \prtset}:\rho\leq\tau)\}=\tblc(\tau),
\end{equation*}
and by Lemma \ref{lem:tblc_r}, $(\tblc(\sigma))_{\prtset}=\tblc(\pi_{\prtset})$.
Therefore, $\tblc(\tau)\subseteq\tblc(\pi_{\prtset})$ and hence, $\tau\leq\pi_{\prtset}$, by Lemma \ref{lem:tblc_dbi},
since $\pi_{\prtset}$ is {\dbi}.

Finally, it is clear that $\pi_{\prtset}=e$, i.e., the set $A_{\prtset}$ is empty, if and only if there do not exist $i<j$ in $\prtset$ such that
$T_{i,j}\leq\sigma$.
\end{proof}

\subsection{}
For $\tau\in S_{\prtt}$ and $\prtset\in \prtt$, define $\tau_{\prtset}\in S_{\prtt}$ by
\begin{equation}\label{def: taux}
\tau_{\prtset}(r)=\begin{cases}
\tau(r) &  r\in \prtset, \\
r & \text{otherwise.}
\end{cases}
\end{equation}
Thus, if $\tau=\iota((\sigma_{\prtset})_{\prtset\in \prtt})$ then $\tau_{\prtset}=\iota_{\prtset}(\sigma_{\prtset})$ for all $\prtset\in \prtt$.

We give a characterization of permutations \dbi.

\begin{proposition} \label{prop: chardbi}
The following two properties are equivalent for $\sigma\in S_n$.
\begin{enumerate}
\item $\sigma$ is \dbi.
\item For every partition $\prtt$ of $[n]$ and $\tau\in S_{\prtt}$ we have
\[
\tau\leq\sigma\iff \tau_{\prtset}\leq\sigma\ \forall \prtset\in \prtt.
\]
\end{enumerate}
\end{proposition}

\begin{observation} \label{obs: ijcyc}
Let $\sigma\in S_n$ and $i\ne j$. Assume that $i$ and $j$ are in the same cycle of $\sigma$.
Then, the cycles of $\sigma'=\sigma T_{i,j}$ are contained in the cycles of $\sigma$,
and $i$, $j$ are in different cycles of $\sigma'$.
\end{observation}

\begin{proof}[Proof of Proposition \ref{prop: chardbi}]
Clearly, for any $\tau\in S_{\prtt}$ and $i\in[n]$
$$\maxperm{\tau}(i)={\max}_{\prtset\in \prtt}\maxperm{\tau_{\prtset}}(i).$$
Hence, if $\sigma\in S_n$ is {\dbi}, then for every $\tau\in S_{\prtt}$ we have
$$
\tau\leq\sigma\iff\tau_{\prtset}\leq\sigma\ \forall \prtset\in \prtt.
$$

Conversely, Suppose that $\sigma$ is not {\dbi}.
Then $\sigma$ does not avoid at least one of the patterns $4231$, $35142$, $42513$, or $351624$.
Equivalently, there are indices $a_s<a_{s+1}<\cdots<a_t$ and
$b_s<b_{s+1}<\cdots<b_t$ in $[n]$ with $s\in\{0,1\}$ and $t\in\{4,5\}$, such that $\sigma(a_1)=b_4$, $\sigma(a_4)=b_1$ and
\begin{itemize}
\item if $s=1$, then $\sigma(a_2)=b_2$; if $s=0$, then $\sigma(a_0)=b_2$ and $\sigma(a_2)=b_0$.
\item if $t=4$, then $\sigma(a_3)=b_3$; if $t=5$, then $\sigma(a_3)=b_5$ and $\sigma(a_5)=b_3$.
\end{itemize}
Denote $A:=\{a_i: s\leq i\leq t\}$.
Let $\tau$ be a permutation such that
\[
\{\tau(a_1),\tau(a_2)\}=\{b_3,b_4\},\ \{\tau(a_3),\tau(a_4)\}=\{b_1,b_2\},\ \tau(a_i)=b_i
\]
for every $i\in [s,t]\setminus [1,4]=\{s,t\}\setminus\{1,4\}$  and $\tau(r)=\sigma(r)$ for every $r\notin A$.

Note that $\tau([a_2])=\sigma([a_2])\cup\{b_3\}\setminus\{b_2\}$ and hence $\#(\tau([a_2])\cap[b_2])=\#(\sigma([a_2])\cap[b_2])-1$.
Therefore, $\tau\nleq\sigma$.

By Observation \ref{obs: ijcyc}, upon replacing $\tau$ by $\tau T_{a_1,a_2}$, $\tau T_{a_3,a_4}$
or $\tau T_{a_1,a_2}T_{a_3,a_4}$ if necessary, we may assume that $a_1,a_2$ lie in different cycles of $\tau$,
and the same for $a_3,a_4$.
Let $\sim$ be an equivalence relation on $[n]$ such that $a_1\not\sim a_2$, $a_3\not\sim a_4$ and $\tau(r)\sim r$ for all $r$.
(In fact, we may choose such $\sim$ with precisely two equivalence classes.)
As usual, let $\prtt$ be the set of equivalence classes of $\sim$.
We claim that $\tau_{\prtset}\leq\sigma$ for all $\prtset\in \prtt$, i.e., $\#(\sigma([i])\cap[j])\leq \#(\tau_{\prtset}([i])\cap[j])$ for every $i,j$.

Since $\prtset$ is $\tau$-invariant,
\begin{align*}
\#(\tau_{\prtset}([i])\cap[j])&=\#(([i]\setminus \prtset)\cap[j])+\#(\tau([i]\cap \prtset)\cap[j])\\
&=\#([\min(i,j)]\setminus \prtset)+\#(\tau([i])\cap[j]\cap \prtset),
\end{align*}
and since $\tau\equiv\sigma$ outside the set $A$,
\begin{equation*}
\#(\sigma([i])\cap[j])=\#(\tau([i])\cap[j])-\#(\tau([i]\cap{A})\cap[j])+\#(\sigma([i]\cap{A})\cap[j]).
\end{equation*}
Hence,
$$\#(\tau_{\prtset}([i])\cap[j])-\#(\sigma([i])\cap[j])=\alpha+\beta$$
where
$$\alpha:=\#([\min(i,j)]\setminus \prtset)-\#(\tau([i])\cap[j]\setminus \prtset)$$
and
$$\beta:=\#(\tau([i]\cap{A})\cap[j])-\#(\sigma([i]\cap{A})\cap[j]).$$
It is easy to verify that $\beta\geq 0$ unless $a_2\leq i<a_3$ and $b_2\leq j<b_3$, in which case $\beta=-1$.
Therefore, it would follow that $\alpha+\beta\geq 0$, as required, provided that we show that $\alpha\geq 0$
and moreover, if $a_2\leq i<a_3$ and $b_2\leq j<b_3$, then $\alpha>0$.

If $i\leq j$, then since $\prtset$ is $\tau$-invariant,
$$\#([\min(i,j)]\setminus \prtset)=\#([i]\setminus \prtset)=\#(\tau([i]\setminus \prtset))=\#(\tau([i])\setminus \prtset)$$
and therefore,
\begin{equation*}
\alpha=\#(\tau([i])\setminus \prtset)-\#(\tau([i])\cap[j]\setminus \prtset)=\#((\tau([i])\setminus[j])\setminus \prtset)\geq0.
\end{equation*}
Moreover, if $a_2\leq i<a_3$ and $b_2\leq j<b_3$, then $\alpha=\#((\tau([i])\setminus[j])\setminus \prtset)>0$
since both $\tau(a_1),\tau(a_2)$ belong to $\tau([i])\setminus[j]$, but they cannot both belong to $\prtset$
since $\tau(a_1)\sim a_1\not\sim a_2\sim\tau(a_2)$.

Similarly, if $i\geq j$, then
\begin{equation*}
\alpha=\#([j]\setminus \prtset)-\#(\tau([i])\cap[j]\setminus \prtset)=\#(([j]\setminus\tau([i]))\setminus \prtset)\geq0.
\end{equation*}
Moreover, if $a_2\leq i<a_3$ and $b_2\leq j<b_3$, then $\alpha=\#(([j]\setminus\tau([i]))\setminus \prtset)>0$ since
$\{\tau(a_3),\tau(a_4)\}\subseteq [j]\setminus\tau([i])$ but $\{\tau(a_3),\tau(a_4)\}\nsubseteq \prtset$.
\end{proof}

\subsection{}
We now give another characterization of smooth permutations.
\begin{proposition} \label{prop: charsm}
The following conditions are equivalent for $\sigma\in S_n$.
\begin{enumerate}
\item $\sigma$ is smooth.
\item For every partition $\prtt$ of $[n]$, the maximum
\[
\sigma_{\prtt}:={\max}_{\leq_{\prtt}}(S_{\prtt}\cap (S_n)_{\leq\sigma})
\]
with respect to $\leq_{\prtt}$ exists.
\end{enumerate}
Moreover, in this case $\sigma_{\prtt}=\iota\left((\sigma[\prtset])_{\prtset\in \prtt}\right)$.
Finally, $\sigma_{\prtt}=e$ (i.e.,
\[
S_{\prtt}\cap(S_n)_{\leq\sigma}=\{e\}),
\]
if and only if there do not exist $i<j$ with $i\sim j$ such that $T_{i,j}\leq\sigma$.
\end{proposition}

We start with the following simple result.
\begin{lemma}\label{lem:coxeter}
Let $\sigma$ be the right cyclic shift in $S_r$, $r>1$.
Then
\begin{enumerate}
\item The interval $(S_r)_{\leq\sigma}=\{\tau\in S_r:\tau\leq\sigma\}$
is isomorphic as a poset to the Boolean lattice $\{0,1\}^{r-1}$.
In particular, the are precisely $r-1$ maximal elements in $(S_r)_{<\sigma}$, namely
$\sigma T_{i,r}$, $i\in [r-1]$.
\item $\tau\geq\sigma$ if and only if $\tau(r)=1$.
\item The following conditions are equivalent for $\sigma'\in S_r$.
\begin{enumerate}
\item \label{part: ge} $\sigma'\geq\tau$ for all $\tau\in(S_r)_{<\sigma}$.
\item \label{part: ge2} $\sigma'\geq\sigma T_{1,r}=(134\ldots r2)$ and $\sigma'\geq\sigma T_{r-1,r}=(23\ldots (r-1)1r)$.
\item \label{part: rr-1} Either $\sigma'(r)=1$ or $\sigma'(r)=2$ and $\sigma'(r-1)=1$.
\item \label{part: minv} $\sigma'\geq\sigma$ or $\sigma'\geq\sigma^2$.
\end{enumerate}
\end{enumerate}
\end{lemma}

\begin{proof}
The first part follows from the fact that $\sigma=T_{1,2}\cdots T_{r-1,r}$ is a Coxeter element
and $\sigma T_{i,r}=T_{1,2}\cdots\hat T_{i,i+1}\cdots T_{r-1,r}$, $i\in [r-1]$.

The second part is clear.

Evidently, \eqref{part: ge}$\implies$\eqref{part: ge2}.

If $\sigma'\geq\sigma T_{1,r}$ and $\sigma'\geq\sigma T_{r-1,r}$, then $\sigma'(r)\in\{1,2\}$ and $\sigma'^{-1}(1)\in\{r-1,r\}$ respectively.
Thus, \eqref{part: ge2}$\implies$\eqref{part: rr-1}.

The implication \eqref{part: rr-1}$\implies$\eqref{part: minv} is straightforward.

It is immediate to verify that $\sigma T_{i,r}\leq\sigma^2$ for all $i\in [r-1]$.
Hence, using the first part, \eqref{part: minv}$\implies$\eqref{part: ge}.
\end{proof}

\begin{lemma}\label{lem:4}
Let $\sigma$ be $4231$ avoiding but not smooth. Then, there exists an index $i$ such that
\[
\maxperm{\sigma}(\maxperm{\sigma^{-1}}(i))>\maxperm{\sigma}(i)>i\text{ and }
\maxperm{\sigma^{-1}}(\maxperm{\sigma}(i))>\maxperm{\sigma^{-1}}(i)>i.
\]
\end{lemma}

\begin{proof}
We first remark that the inequality $\maxperm{\sigma}(\maxperm{\sigma^{-1}}(i))>\maxperm{\sigma}(i)$
implies that $\maxperm{\sigma}(i)>i$ (or equivalently, $\maxperm{\sigma^{-1}}(i)>i$).

Since $\sigma$ is $4231$ avoiding but not smooth, it is not covexillary, i.e., the set
$$P:=\{(a,b,c,d): a<b<c<d\text{ and }\sigma(c)<\sigma(d)<\sigma(a)<\sigma(b)\}$$
is non-empty.
Choose $(a,b,c,d)\in P$ such that $(a,b,\sigma(c),\sigma(d))$ is minimal with respect to the lexicographic order
(from left to right) on $[n]^4$.
Since $\sigma$ is $4231$ avoiding, either $a=1$ or $\maxperm{\sigma}(a-1)<\sigma(b)$.
Moreover, if there is $\tilde{a}<a$ such that $\sigma(d)<\sigma(\tilde{a})<\sigma(b)$ then $(\tilde{a},b,c,d)\in P$,
contradicting the minimality of $(a,b,\sigma(c),\sigma(d))$.
It follows that either $a=1$ or $\maxperm{\sigma}(a-1)<\sigma(d)$.
In particular, $\sigma([a-1]\cup\{c\})\subseteq[\sigma(d)-1]$ and hence,
\begin{equation}\label{eq:ad}
a<\sigma(d).
\end{equation}
If there is $a<\tilde{b}<b$ such that $\sigma(\tilde{b})>\sigma(a)$, then $(a,\tilde{b},c,d)\in P$,
gainsaying the minimality of $(a,b,\sigma(c),\sigma(d))$.
It follows, since $\maxperm{\sigma}(a-1)<\sigma(d)<\sigma(a)$ (or $a=1$), that
\begin{equation}\label{eq:btoa}
\maxperm{\sigma}(b-1)=\sigma(a).
\end{equation}
Similarly,
\begin{equation}\label{eq:cb}
\sigma(c)<b
\end{equation}
and
\begin{equation}\label{eq:dtoc}
\maxperm{\sigma^{-1}}(\sigma(d)-1)=c.
\end{equation}
Let $i=\max\left(a,\sigma(c)\right)$. It follows from \eqref{eq:ad} and \eqref{eq:cb} that $i<\min\left(b,\sigma(d)\right)$.
Therefore, by \eqref{eq:btoa}, $\maxperm{\sigma}(i)\leq\maxperm{\sigma}(b-1)=\sigma(a)$.
Hence, since $a\leq i$, it follows that
\begin{equation*}\label{eq:itoa}
\maxperm{\sigma}(i)=\sigma(a).
\end{equation*}
Similarly,
\begin{equation*}\label{eq:itoc}
\maxperm{\sigma^{-1}}(i)=c.
\end{equation*}
The lemma now follows by noting that
$\maxperm{\sigma}(c)\geq \sigma(b)>\sigma(a)$
and
$\maxperm{\sigma^{-1}}(\sigma(a))\geq \sigma^{-1}(\sigma(d))=d>c$.
\end{proof}

\begin{proof}[Proof of Proposition \ref{prop: charsm}]

Suppose that $\sigma$ is smooth. Then, it follows immediately from Proposition \ref{prop: chardbi}
and Lemma \ref{lem:1perm} that for any partition $\prtt$ of $[n]$ and any $\tau\in S_{\prtt}$ we have
$\tau\leq\sigma$ if and only if $\tau_{\prtset}\leq\iota_{\prtset}(\sigma[\prtset])$ for all $\prtset$.
Thus,
\[
\iota((\sigma[\prtset])_{\prtset\in \prtt})={\max}_{\leq_{\prtt}}(S_{\prtt}\cap (S_n)_{\leq\sigma})
\]
as required.

Suppose now that $\sigma$ is not \dbi.
Then, by Proposition \ref{prop: chardbi}, there exists a partition $\prtt$ of $[n]$
and $\tau\in S_{\prtt}$ such that $\tau\not\leq\sigma$ but $\tau_{\prtset}\leq\sigma$ for all $\prtset\in \prtt$.
Since $\tau=\vee_{\prtset\in \prtt}\tau_{\prtset}$ in $S_{\prtt}$ with respect to $\leq_{\prtt}$, it follows that
the set $(S_n)_{\leq\sigma}\cap S_{\prtt}$ does not admit a maximum with respect to $\leq_{\prtt}$.

It remains to consider the case where $\sigma$ is \dbi, and in particular $4231$ avoiding, but not smooth.
Let $i$ be as in Lemma \ref{lem:4} and let $j=\maxperm{\sigma}(i)$ and $k=\maxperm{\sigma^{-1}}(i)$. Then,
$\maxperm{\sigma}(k)>j>i$ and $\maxperm{\sigma^{-1}}(j)>k>i$.

Upon passing to $\sigma^{-1}$ if necessary, we may assume without loss of generality that $j\leq k$.
Let $A=\{i,j,j+1,\ldots,k+1\}$.
Then, $\cycr_{A\setminus\{i\}}, \cycr_{A\setminus\{k+1\}}\leq\sigma$ (by \eqref{eq:cyc1<=})
but $\cycr_{A},\cycr_{A}^2\not\leq\sigma$ since $\maxperm{\cycr_A^{-1}}(i)=k+1>\maxperm{\sigma^{-1}}(i)$
and $\maxperm{\cycr_A^2}(i)=j+1>\maxperm{\sigma}(i)$.
Let $\prtt$ be the partition of $[n]$ consisting of $A$ and the singletons $\{r\}$, $r\notin A$.
Note that $S_{\prtt}\subseteq S_n$ is isomorphic to $S_{k-j+3}$.
(The order on $S_{\prtt}$ induced from the Bruhat order on $S_n$ coincides with $\leq_{\prtt}$.)
By Lemma \ref{lem:coxeter}, there is no $\sigma'\in S_{\prtt}$ such that $\cycr_{A\setminus\{i\}},
\cycr_{A\setminus\{k+1\}}\leq\sigma'\leq\sigma$. It follows that ${\max}_{\leq_{\prtt}}(S_{\prtt}\cap (S_n)_{\leq\sigma})$ does not exist.
\end{proof}

\begin{example}
Let $\sim$ be the equivalence relation
\[
i\sim j\iff i\equiv j\pmod2
\]
on $[n]$ and let $\iota:S_{n_1}\times S_{n_2}\rightarrow S_n$ be the corresponding embedding where $n_1=\lfloor\frac n2\rfloor$
and $n_2=\lceil\frac n2\rceil$.
Then, for every smooth $\sigma\in S_n$
\[
S_{\prtt}\cap(S_n)_{\leq\sigma}=\{e\}\iff \sigma\text{ is $321$ avoiding}.
\]
Indeed, $\sigma$ is $321$ avoiding if and only if $\sigma$ is a product of distinct simple reflections if and only if
$\tbl_{\Trans}(\sigma)$ consists of simple reflections.
\end{example}

\section{Relation to Dyck paths}\label{sec:sigma_fg_bijection}

In this section we prove Theorem \ref{thm:sigma_fg_bijection}.

\subsection{}
For any $n\geq 1$, let
\begin{equation*}
\Dyck_n=\{f:[n]\to[n]: f \text{ is weakly increasing, }f(i)\geq i \text{ for all } i\in [n]\}.
\end{equation*}

We can view elements of $\Dyck_n$ as Dyck paths from $(0,0)$ to $(n,n)$ by taking
$f(i)$ to be the minimal $x$ such that the lattice point $(x,i)$ lies in the path.

We can give an alternative interpretation of $\Dyck_n$ as follows.
For any subset $\Gamma\subseteq\Trans$ define $f^{\ast}_\Gamma:[n]\rightarrow[n]$ by
\begin{equation*}
f^{\ast}_\Gamma(i)=\max\Big(\{i\}\cup\{j>i:T_{i,j}\in \Gamma\}\Big).
\end{equation*}
For any $f\in\Dyck_n$, let $\Lambda_f:=\{(i,j)\in[n]\times[n]:i<j\leq f(i)\}$.

\begin{observation}\label{obs:hat1}
The map $\Gamma\mapsto f^{\ast}_\Gamma$ is a bijection between the downward closed subsets of $\Trans$ and $\Dyck_n$.
The inverse map is given by
\[
f\mapsto\{T_{i,j}:(i,j)\in\Lambda_f\}.
\]
\end{observation}

\begin{observation}\label{obs:hat2}
Let $\Gamma$ be a downward closed subset of $\Trans$ and let $1\leq i\leq l\leq i+1$.
Then, $f^{\ast}_\Gamma(f^{\ast}_\Gamma(i))>f^{\ast}_\Gamma(l)$ if and only if there are $i<j<k$ such that $T_{i,j}, T_{j,k}\in \Gamma$ but $T_{l,k}\notin \Gamma$.
\end{observation}

\subsection{}
Given a Dyck path $f\in\Dyck_n$ we define a \emph{\dcr}\ of $f$ to be
a function $g:[n]\rightarrow\{0,1\}$ such that
\begin{enumerate}
\item $g(i)=0$ whenever $f(f(i))=f(i)$.
\item $g(i)=g(i+1)$ whenever $i<n$ and $f(i+1)<f(f(i))$.
\end{enumerate}
In particular,
\begin{equation}\label{eq:gn}
g(n)=0 \text{ and if } n>1 \text{ then } g(n-1)=0 \text{ as well}.
\end{equation}
Note that the number of \dcr s of $f$ is
\begin{equation}\label{eq:g_counting}
2^{\#\{i\in [n-1]:f(i)<f(f(i))=f(i+1)\}}.
\end{equation}

We say that a vertex $(p,q)$ of a Dyck path $\mathcal{D}$ is distinguished if it is the top left corner
of a (non-degenerate) rectangle $R$ such that
\begin{enumerate}
\item The left side of $R$ is the intersection of $\mathcal{D}$ with the vertical line $x=p$.
\item The top side of $R$ is contained in $\mathcal{D}$.
\item The bottom right corner of $R$ lies on the main diagonal $x=y$.
\end{enumerate}
If $f\in\Dyck_n$ corresponds to $\mathcal{D}$ then the exponent in \eqref{eq:g_counting}
is precisely the number of distinguished vertices of $\mathcal{D}$.
Thus, \eqref{eq:g_counting} counts the number of (unrestricted) $2$-colorings of the set of distinguished vertices.

Denote by $\Pairs_n$ the set of pairs $(f,g)$ consisting of a function $f\in\Dyck_n$
and a \dcr\ $g$ of $f$.
Informally, $\Pairs_n$ is the set of \dcrt\ Dyck paths.

For any $A\subseteq\Special$ define $f_A:[n]\rightarrow[n]$ and $g_A:[n]\rightarrow\{0,1\}$ by
\begin{align*}
f_A&=f^{\ast}_{A_\Trans},\\
g_A(i)&=\begin{cases}1&\text{if $i<f_A(i)<n$ and }R_{i,f_A(i),f_A(i)+1}\in A,\\0&\text{otherwise.}\end{cases}
\end{align*}

\begin{lemma}\label{lem:gLR}
Let $A\subseteq\Special$ be {\adm}, let $f=f_A$ and let $i,j\in[n]$ be such that $i<j\leq f(i)<f(j)$. Then,
\begin{equation}\label{eq:ijk_gen}
\text{for any } f(i)<k\leq f(j) \text{ exactly one of the permutations } R_{i,j,k},L_{i,j,k}\text{ belongs to } A.
\end{equation}
More precisely,
\begin{equation*}
g_A(i)=1\iff L_{i,j,k}\notin A\text{ and }R_{i,j,k}\in A;\quad
g_A(i)=0\iff L_{i,j,k}\in A\text{ and }R_{i,j,k}\notin A.
\end{equation*}
\end{lemma}

\begin{proof}
If $i<j\leq f(i)<k\leq f(j)$, then $T_{i,j},T_{j,k}\in A$ but $T_{i,k}\notin A$ and therefore \eqref{eq:ijk_gen} follows from \eqref{item: adm2} and \eqref{item: adm3}.

Suppose that $g_A(i)=1$, i.e., $R_{i,f(i),f(i)+1}\in A$.
Then $L_{i,j,f(i)+1}\notin A$ by \eqref{item: adm2}, since $T_{i,f(i)+1}\notin A$.
Therefore, $L_{i,j,f(j)}\notin A$ by \eqref{item: adm1} and hence $R_{i,j,f(j)}\in A$, by \eqref{eq:ijk_gen}.
Therefore,  by \eqref{item: adm1}, $R_{i,j,k}\in A$ for every $j<k\leq f(j)$.
In particular, for every $f(i)<k\leq f(j)$, $R_{i,j,k}\in A$ and hence $L_{i,j,k}\notin A$, by \eqref{eq:ijk_gen}.

Similarly, if $g_A(i)=0$, i.e., $R_{i,f(i),f(i)+1}\notin A$ then $L_{i,f(i),f(i)+1}\in A$ by \eqref{eq:ijk_gen}. Hence $R_{i,j,f(i)+1}\notin A$ by \eqref{item: adm2}, since $T_{i,f(i)+1}\notin A$. Therefore $R_{i,j,f(j)}\notin A$ by \eqref{item: adm1} and hence $L_{i,j,f(j)}\in A$ by \eqref{eq:ijk_gen}.
Therefore,  by \eqref{item: adm1}, $L_{i,j,k}\in A$ for every $j<k\leq f(j)$.
In particular, for every $f(i)<k\leq f(j)$, $L_{i,j,k}\in A$ and hence $R_{i,j,k}\notin A$, by \eqref{eq:ijk_gen}.
\end{proof}

Conversely, for every pair of functions $f:[n]\rightarrow[n]$ and $g:[n]\rightarrow\{0,1\}$ define
\begin{align*}
A_{f,g}=&\{T_{i,j}:(i,j)\in\Lambda_f\}\cup\\&\{R_{i,j,k}:(i,j)\in\Lambda_f,\, (j,k)\in\Lambda_f \text{ and if } (i,k)\notin\Lambda_f \text{ then }g(i)=1  \}\cup\\
&\{L_{i,j,k}:(i,j)\in\Lambda_f,\, (j,k)\in\Lambda_f  \text{ and if } (i,k)\notin\Lambda_f \text{ then }g(i)=0 \}.
\end{align*}

Denote by $\Adm=\Adm_n$ the set of {\adm} subsets of $\Special$.

\begin{proposition}\label{prop:Afg_bijection}
The map $A\mapsto (f_A,g_A)$ is a bijection $\Adm\rightarrow\Pairs$ whose inverse is $(f,g)\mapsto A_{f,g}$.
\end{proposition}

\begin{proof}
Suppose that $A$ is \adm.

Let $f=f_A$ and $g=g_A$. We show that $A_{f,g}=A$.
Clearly, $T_{i,j}\in A$ if and only if $j\leq f(i)$, and moreover, $R_{i,j,k},L_{i,j,k}\in A$ if $k\leq f(i)$.
Suppose that $k>f(i)$. Then, by \admis,
$\{R_{i,j,k},L_{i,j,k}\}\cap A\ne\emptyset$ if and only if $j\leq f(i)$ and $k\leq f(j)$.
Moreover, in this case, by the last part of Lemma \ref{lem:gLR}, if $g_A(i)=1$, then $L_{i,j,k}\notin A, R_{i,j,k}\in A$ and if $g_A(i)=0$, then
$L_{i,j,k}\in A, R_{i,j,k}\notin A$.

We now show that $(f,g)\in\Pairs$. By Observation \ref{obs:hat1}, $f\in\Dyck_n$.
It is also clear that $g(i)=0$ if $f(f(i))=f(i)$ since $i<f(i)<n$ and $R_{i,f(i),f(i)+1}\in A$
would imply that $T_{f(i),f(i)+1}\in A$ and hence $f(f(i))>f(i)$.
Suppose that $f(i+1)<f(f(i))$ and let $j=f(i)$ and $k=f(j)$.
Then, $i<i+1<j=f(i)\leq f(i+1)<k<f(j)$.
Therefore, if $g(i)=1$, then $R_{i,j,k}\in A$ by Lemma \ref{lem:gLR}, hence $R_{i+1,j,k}\in A$ by \eqref{item: adm1},
and therefore $g(i+1)=1$ again by Lemma \ref{lem:gLR}.
Similarly,  if $g(i)=0$, then $L_{i,j,k}\in A$, hence $L_{i+1,j,k}\in A$ and therefore $g(i+1)=0$.
Thus, $(f,g)\in\Pairs$.

On the other hand, let $(f,g)\in\Pairs$ and $A=A_{f,g}$. It is easy to check that $A$ is {\adm} and $(f_A,g_A)=(f,g)$.
The only non-trivial observation to make is that if $i+1<j\leq f(i)\leq f(i+1)<k\leq f(j)$, then $f(i+1)<f(j)\leq f(f(i))$,
hence $g(i)=g(i+1)$ and therefore, by Lemma \ref{lem:gLR}, if $L_{i,j,k}\in A$, then $L_{i+1,j,k}\in A$ and if $R_{i,j,k}\in A$, then $R_{i+1,j,k}\in A$.
\end{proof}

\begin{corollary}
Let $\Gamma$ be a (possibly empty) downward closed subset of $\Trans$.
Then,
$$\#\{A\in{\Adm}_n: A_\Trans=\Gamma\}=2^r$$
where $r$ is the number of indices $i<n$ satisfying
the following two properties.
\begin{enumerate}
\item There exists $i<j<k$ such that $T_{i,j}, T_{j,k}\in \Gamma$ but $T_{i,k}\notin \Gamma$.
\item For every $i<j<k$ such that $T_{i,j}, T_{j,k}\in \Gamma$ we have $T_{i+1,k}\in \Gamma$.
\end{enumerate}
\end{corollary}
\begin{proof}
By Observation \ref{obs:hat1},
\begin{align*}
\{(f,g)\in\Pairs_n: (A_{f,g})_\Trans=\Gamma\}&=\{(f,g)\in\Pairs_n: \{T_{i,j}:i<j\leq f(i)\}=\Gamma\}\\
&=\{(f,g)\in\Pairs_n: f=f^{\ast}_\Gamma\}.
\end{align*}
Therefore, by Proposition \ref{prop:Afg_bijection}
\begin{align*}
\#\{A\in{\Adm}_n: A_\Trans=\Gamma\}&=\#\{(f,g)\in\Pairs_n: (A_{f,g})_\Trans=\Gamma\}\\
&=\#\{g:[n]\to\{0,1\}: (f^{\ast}_\Gamma,g)\in\Pairs_n\},
\end{align*}
and the result follows from \eqref{eq:g_counting} and Observation \ref{obs:hat2}.
\end{proof}

\subsection{}
Given $f\in\Dyck_n$ and $g:[n]\rightarrow\{0,1\}$, write $g^{-1}(0)=\{i_1,\ldots,i_k\}$ and  $g^{-1}(1)=\{j_1,\ldots,j_l\}$
with $i_1<\cdots<i_k$ and $j_1<\cdots<j_l$, and define
\begin{equation} \label{eq: redsmooth}
\sigma(f,g)=\lshft{j_1}{f(j_1)}\cdots\lshft{j_l}{f(j_l)}\rshft{i_k}{f(i_k)}\cdots\rshft{i_1}{f(i_1)}.
\end{equation}

\begin{observation}\label{obs:sigma_fg}
Suppose that $g$ is a \dcr\ of $f$ and $i\in[n]$ is such that $f(i)>i$ and $f(i-1)=i-1$ (or $i=1$).
Then, $\sigma(f,g)([i-1])=[i-1]$ and if $g(i)=0$, then $\sigma(f,g)(f(i))=i$.

In particular, if $i$ is the minimal index such that $j:=f(i)>i$, then $\sigma(r)=r$ for every $r<i$ and if $g(i)=0$, then $\sigma(f,g)(j)=i<\sigma(f,g)(j-1)$.
\end{observation}

\begin{lemma}\label{lem:sigma_fg}
Suppose that $f\in\Dyck_n$ and $g$ is a \dcr\ of $f$.
Let $i\in[n]$ be such that $j:=f(i)>i$, $g(i)=0$ and $f(i-1)=i-1$ (or $i=1$).
Define a function $f'\in\Dyck_n$ by $f'(i)=f(i)-1$ and $f'(r)=f(r)$ for every $r\neq i$. Then,
\begin{enumerate}
\item $g$ is a \dcr\ of $f'$.
\item $T_{i,j}$ is a wedge for $A_{f,g}$ and the derived set is $A_{f',g}$. (See \S\ref{sec:wedges}.)
\item $\sigma(f',g)=\sigma(f,g)T_{j-1,j}$.
\end{enumerate}
\end{lemma}

\begin{proof}
It is clear that $(f',g)\in\Pairs$ since $f'(i)>i-1=f'(f'(i-1))$ and the condition $f'(i+1)<f'(f'(i))$ implies $f(i+1)<f(f(i))$.
Note that $(i,j)\in\Lambda_f$ , $(i-1,i)\notin\Lambda_f$ (or $i=1$) and $(i,k)\notin\Lambda_f$ for every $k>j$.
Therefore, $T_{i,j}\in A_{f,g}$, $T_{i-1,i}\notin A_{f,g}$ (or $i=1$) and since $g(i)=0$ also $R_{i,j,j+1}\notin A_{f,g}$ (or $j=n$).
Hence, $T_{i,j}$ is a wedge for $A_{f,g}$.
Noting that $\Lambda_{f`}=\Lambda_f\setminus\{(i,j)\}$, it is elementary to check that the derived set of $A_{f,g}$ with respect to $T_{i,j}$
is $A_{f',g}$.
Finally, the last assertion follows readily from the definition \eqref{eq: redsmooth}.
\end{proof}

Let $f\in\Dyck_n$.
Note that by monotonicity, if $f(i+1)<f(f(i))$, then $f(i)<f(f(i))$ and $f(i+1)<f(f(i+1))$.
Therefore, for any \dcr\ $g$ of $f$ we can define the inverse \dcr\ $\tilde g$ by
\[
\tilde g(i)=\begin{cases}0&\text{if }f(f(i))=f(i),\\1-g(i)&\text{otherwise.}\end{cases}
\]
\begin{lemma}\label{lem:tilde_g}
For every $(f,g)\in\Pairs$,
$$\sigma(f,g)^{-1}=\sigma(f,\tilde g).$$
\end{lemma}

\begin{proof}
Clearly, $\sigma(f,g)^{-1}=\sigma(f,1-g)$. Note that $(f,1-g)$ is generally not in $\Pairs$.
However, we claim that $\sigma(f,1-g)=\sigma(f,\tilde g)$.
This follows from the definition of $\sigma(f,g)$ in \eqref{eq: redsmooth}.
The point is that if $f^{-1}(\{j\})=[i,j]$, then
$\rshft jj\cdots\rshft {i+1}j\rshft ij$ is an involution (namely, $r\to i+j-r$ for $i\leq r\leq j$ and $r\to r$ otherwise)
and it commutes with $\rshft rs$ (and with its inverse $\lshft rs$) for every $j<r\leq s$.
Hence, we can flip $1-g$ in the set $\cup_{j:f(j)=j}f^{-1}(\{j\})$ (where it differs from $\tilde g$)
without changing $\sigma(f,1-g)$.
\end{proof}

For every $f\in \Dyck_n$,  let $\ell(f)=\sum_{i=1}^n(f(i)-i)$.

\begin{proposition}\label{propos:sigma_fg}
For every $(f,g)\in\Pairs$,
$$\sigma(f,g)=\pi(A_{f,g}).$$
Moreover, for every $(f,g)\in\Pairs$,
$$\ell(\sigma(f,g))=\ell(f),$$
and hence, since evidently $\ell(f)=\sum_{i=1}^n\ell(\rshft{i}{f(i)})$, the expression \eqref{eq: redsmooth} is reduced.
\end{proposition}

\begin{proof}
We prove it by induction on $\ell(f)$. The case $\ell(f)=0$ is obvious. For the induction step,
let $i$ be the minimal index such that $f(i)>i$. By Lemma \ref{lem:tilde_g}, we may assume without loss of generality
that $g(i)=0$. Let $j=f(i)$ and let $f'$ be defined as in Lemma \ref{lem:sigma_fg}.
By the induction hypothesis, $\pi(A_{f',g})=\sigma(f',g)$ and $\ell(\sigma(f',g))=\ell(f)$ and by the third part of Lemma \ref{lem:sigma_fg}, $\sigma(f',g)=\sigma(f,g)T_{j-1,j}$.
Therefore, by the second part of Lemma \ref{lem:pi} and the second part of Lemma \ref{lem:sigma_fg},
$$
\pi(A_{f,g})=\pi\left((A_{f,g})'\right)T_{j-1,j}=\pi(A_{f',g})T_{j-1,j}=\sigma(f',g)T_{j-1,j}=\sigma(f,g).
$$
Moreover, $\sigma(f,g)(j)<\sigma(f,g)(j-1)$ by Observation \ref{obs:sigma_fg} and hence,
\begin{equation*}
\ell(\sigma(f,g))=\ell(\sigma(f,g)T_{j-1,j})+1=\ell(\sigma(f',g))+1=\ell(f')+1=\ell(f).
\qedhere\end{equation*}
\end{proof}

Theorem \ref{thm:sigma_fg_bijection} follows directly by combining Theorem \ref{thm:bijection} and Propositions \ref{prop:Afg_bijection} and \ref{propos:sigma_fg}.

\begin{question}
Can we describe explicitly the partial order on $\Pairs$ induced
from the Bruhat order on $(S_n)_{\smth}$ by the map $\sigma$ ?
\end{question}

\section{Enumerative consequences}\label{sec:enumeration}

In this section we interpret combinatorial properties of smooth permutations in terms
of the bijection of the previous section, and recover some known enumerative results.

\begin{proposition}\label{propos:sigma_fg_bijection}
Let $(f,g)\in\Pairs_n$. Then,
\begin{enumerate}
\item $\sigma(f,g)$ is $231$ avoiding (also known as stack-sortable in Knuth's terminology)
if and only if $g\equiv0$.
\item $\sigma(f,g)$ is $321$ avoiding if and only if $f(i)\leq i+1$ for all $i\in [n-1]$.
\item $\sigma(f,g)$ is indecomposable\footnote{Recall that a permutation $\sigma$ in $S_n$ is called indecomposable
if there does not exist $1\leq k<n$ such that $\sigma([k])=[k]$.} if and only if $f(i)>i$ for all $i\in [n-1]$.
\end{enumerate}
\end{proposition}

\begin{proof}~
\begin{enumerate}
\item
We first show that if $\sigma\in S_n$ and $i<j<k$ are such that $R_{i,j,k}\leq\sigma$ but $T_{i,k}\not\leq\sigma$, then
$\sigma$ is not $231$ avoiding.
Indeed, by our conditions, $\maxperm{\sigma^{-1}}(i)\geq k$, $\maxperm{\sigma}(j)\geq k$ and $j\leq \maxperm{\sigma}(i)<k$.
Therefore, there are $a\leq i$, $i<b\leq j$ and $c\geq k$ such that
$j\leq\sigma(a)<k$, $\sigma(b)\geq k$ and $\sigma(c)\leq i$. Then, $a<b<c$ and $\sigma(c)<\sigma(a)<\sigma(b)$ as claimed.

It follows that if $g\not\equiv0$, then $\sigma(f,g)$ is not $231$ avoiding.
Indeed, suppose that $g(i)=1$.
Then clearly $i<f(i)<f(f(i))$ and then $R_{i,f(i),f(f(i))}\in A_{f,g}$ but $T_{i,f(f(i))}\notin A_{f,g}$,
i.e., $R_{i,f(i),f(f(i))}\leq\sigma(f,g)$ but $T_{i,f(f(i))}\nleq\sigma(f,g)$.

Conversely, we show by induction on $\ell(f)$ that for every $f\in\Dyck_n$, $\sigma:=\sigma(f,0)$ is $231$ avoiding.
This is clear if $\ell(f)=0$. Otherwise, let $i$ be the minimal index such that $j:=f(i)>i$,
let $f'$ be defined as in Lemma \ref{lem:sigma_fg} and let $\sigma'=\sigma(f',0)$.
Then, $\sigma(j)=i$ by Observation \ref{obs:sigma_fg}, $\sigma=\sigma' T_{j-1,j}$, by Lemma \ref{lem:sigma_fg}
and $\sigma'$ is $231$ avoiding by the induction hypothesis. In addition, it is easy to see from the definition \eqref{eq: redsmooth} that
\begin{equation}\label{eq:sigma-1i+1}k:=\sigma^{-1}(i+1)=\begin{cases}
j-1 & f(i+1)=j \\
f(i+1) & f(i+1)>j.
\end{cases}\end{equation}
Assume on the contrary that $\sigma$ is not $231$ avoiding, i.e., there are $a<b<c$ such that $\sigma(c)<\sigma(a)<\sigma(b)$. Then $\sigma'(T_{j-1,j}(c))<\sigma'(T_{j-1,j}(a))<\sigma'(T_{j-1,j}(b))$ and hence, since $\sigma'$ is $231$ avoiding, necessarily $b=j-1$ and $c=j$.
Hence, $a<j-1$ and $i<\sigma(a)<\sigma(j-1)$.
In particular, $\sigma(j-1)>i+1$ and hence $k=f(i+1)>j$, by \eqref{eq:sigma-1i+1}.
In particular, $a\neq k$, i.e., $\sigma(a)\neq i+1$ and hence $\sigma(a)>i+1=\sigma(k)$.
Therefore, $a<j<k$ and $\sigma(k)<\sigma(a)<\sigma(j-1)$, i.e., $\sigma'(k)<\sigma'(a)<\sigma'(j)$, in contradiction to the fact
that $\sigma'$ is $231$ avoiding.

\item Suppose that $f(i)\leq i+1$ for all $i$. We show that $\sigma(f,g)$ is $321$ avoiding by induction on $\ell(f)$.
This is certainly true if $\ell(f)=0$. For the induction step, let $i$ be the smallest index such that $f(i)>i$ and let
$f'$ be defined as in Lemma \ref{lem:sigma_fg}.
By passing to $\tilde g$ if necessary we may assume that $g(i)=0$. Let $\sigma=\sigma(f,g)$ and $\sigma'=\sigma(f',g)$. Then $\sigma=\sigma'T_{i,i+1}$
by Lemma \ref{lem:sigma_fg}. By the induction hypothesis $\sigma'$ is $321$ avoiding, and since $\sigma(r)=r$ for all $r<i$
and $\sigma(i+1)=i$ by Observation \ref{obs:sigma_fg}, it is easy to check that $\sigma$ is $321$ avoiding as well.

Conversely, suppose that $\sigma=\sigma(f,g)$ is $321$ avoiding. We show that induction on $\ell(f)$ that $f(i)\leq i+1$ for all $i$.
Again, the case $\ell(f)=0$ is trivial.
For the induction step, let $i$ be the minimal index such that $j:=f(i)>i$,
let $f'$ be defined as in Lemma \ref{lem:sigma_fg}, let $\sigma'=\sigma(f',g)$
and assume, as we may, that $g(i)=0$.
Then, $\sigma([i-1])=[i-1]$ and $\sigma(j)=i<\sigma(j-1)$ by Observation \ref{obs:sigma_fg} and $\sigma=\sigma' T_{j-1,j}$ by Lemma \ref{lem:sigma_fg}.
It is clear that $\sigma'$ is $321$ avoiding since $\sigma$ is, and $\sigma'(j)>\sigma'(j-1)$.
Thus, by the induction hypothesis $f'(r)\leq r+1$ for all $r$.
Hence, $f(r)\leq r+1$ for every $r\neq i$ and $f(i)\leq i+2$.
If $f(i)=i+2$ then $\sigma(i+2)=i$ and it is easy to see from the definition \eqref{eq: redsmooth} that $\sigma(i+1)=i+1$ and $\sigma(i)\geq i+2$,
in contradiction to the fact that $\sigma'$ is $321$ avoiding.
Thus, $f(r)\leq r+1$ for all $r$.

\item Clearly, $\sigma\in S_n$ is indecomposable if and only if $T_{i,i+1}\leq\sigma$ for all $i<n$, i.e.,
if and only if $f_{\tbl(\sigma)}(i)=f^{\ast}_{\tbl_\Trans(\sigma)}(i)>i$ for all $i<n$.
\qedhere\end{enumerate}
\end{proof}

It is well known that $\#\Dyck_n$, $n\geq1$ is the $n$-th Catalan number $C_n=\frac{1}{n+1}\binom{2n}{n}$.
Since clearly $\#\{(f,g)\in\Pairs_n: g\equiv0\}=\#\Dyck_n$,
combining Theorem \ref{thm:sigma_fg_bijection} and the first part of Proposition \ref{propos:sigma_fg_bijection}
we recover the standard fact that the number of $231$ avoiding permutations in $S_n$ (which are automatically smooth) is $C_n$.

Similarly, Theorem \ref{thm:sigma_fg_bijection} and Proposition \ref{propos:sigma_fg_bijection} enable us to recover several additional enumerative results concerning smooth permutations, as we show next.

\begin{proposition}\label{propos:fibonacci}
for every $n\geq 1$,
\begin{equation}\label{eq:fibonacci}
\#\{(f,g)\in\Pairs_n:f(i)\leq i+1 \text{ for every } i\in [n-1]\}=F_{2n-1},
\end{equation}
where $F_k$, $k\geq1$ is the Fibonacci sequence $F_1=F_2=1$, $F_k=F_{k-1}+F_{k-2}$, $k>2$.
\end{proposition}

\begin{proof}
For every $n\geq 1$, denote
\begin{equation*}
{\OPairs}_n:=\{(f,g)\in\Pairs_n:f(i)\leq i+1 \text{ for every } i\in [n-1]\},
\end{equation*}
and for every $n>1$, let
\begin{equation*}
{\EPairs}_n:=\{(f,g)\in{\OPairs}_n : f(n-1)=n\}.
\end{equation*}
For every $n>1$, the map $(f,g)\mapsto(f|_{[n-1]},g|_{[n-1]})$ is clearly a bijection
${\OPairs}_n\setminus{\EPairs}_n\to{\OPairs}_{n-1}$ and hence
\begin{equation}\label{eq:fibonacci_odd}
\#{\OPairs}_n=\#{\EPairs}_n+\#{\OPairs}_{n-1}.
\end{equation}
Let now $n>2$. For any $f\in\Dyck_n$ such that $f(n-2)=n-1$, let $f^{\dagger}\in\Dyck_{n-1}$ be defined by $f^{\dagger}\equiv f$
on $[n-2]$ and $f^{\dagger}(n-1)=n-1$.
For any $g:[n]\to\{0,1\}$  let $g^{\ddagger}:[n-1]\to\{0,1\}$ be defined by $g^{\ddagger}\equiv g$ on $[n-3]$
and $g^{\ddagger}(n-2)=g^{\ddagger}(n-1)=0$.
It is easy to verify, using \eqref{eq:gn}, that for any $n>2$, the map
\[
(f,g)\mapsto (g(n-2),(f^{\dagger},g^{\ddagger}))
\]
is a bijection $\{(f,g)\in{\EPairs}_n:f(n-2)=n-1\}\to\{0,1\}\times{\EPairs}_{n-1}$.
Also, it is clear that the map $(f,g)\mapsto (f|_{[n-2]},g|_{[n-2]})$ is a bijection $\{(f,g)\in{\EPairs}_n:f(n-2)=n-2\}\to{\OPairs}_{n-2}$.
Therefore, for any $n>2$,
\begin{equation*}
\#{\EPairs}_n= 2\#{\EPairs}_{n-1}+\#{\OPairs}_{n-2}
\end{equation*}
and hence, using \eqref{eq:fibonacci_odd}
\begin{equation}\label{eq:fibonacci_even}
\#{\EPairs}_n=2\#{\EPairs}_{n-1}+(\#{\OPairs}_{n-1}-\#{\EPairs}_{n-1})=\#{\OPairs}_{n-1}+\#{\EPairs}_{n-1}
\end{equation}
Since obviously $\#{\OPairs}_1=\#{\EPairs}_2=1$, it follows from \eqref{eq:fibonacci_odd} and \eqref{eq:fibonacci_even} that
$$
\#{\OPairs}_1,\#{\EPairs}_2,\#{\OPairs}_2,\#{\EPairs}_3,\ldots
$$
is the Fibonacci sequence $(F_k)_{k=1}^{\infty}$.
In particular $\#{\OPairs}_n=F_{2n-1}$ for every $n\geq 1$.
\end{proof}

\begin{corollary}[\cite{MR1417303, MR1603806}]
The number of $321$ avoiding smooth permutations in $S_n$ is $F_{2n-1}$.
\end{corollary}

For every $n\geq 1$, let
\begin{equation*}
\tilde\Dyck_n=\{f\in\Dyck_n: f(i)> i \text{ for all } i\in [n-1]\}.
\end{equation*}
It is well known that for every $n\geq 1$,
\begin{equation}\label{eq:catalan}
\#\tilde\Dyck_{n}=C_{n-1}.
\end{equation}

\begin{proposition}\label{propos:smooth_recursion}
For every $n\geq 1$, let
\[
\tilde{\Pairs}_n:=\{(f,g)\in\Pairs:f\in\tilde\Dyck_n\},\ \ \tilde p_n=\#\tilde{\Pairs}_n.
\]
Then, for every $n\geq 2$,
\begin{equation}\label{eq:recursion_a}
\tilde p_n=\tilde p_{n-1}+2\sum_{i=1}^{n-2}C_{i-1}\tilde p_{n-i}.
\end{equation}
Consequently, the generating function $\sum_{n=1}^\infty\tilde p_n x^n$  is
\begin{equation}
\label{eq:generating_a}
\left(\frac{1}{x}-\frac{1}{\sqrt{1-4x}}\right)^{-1},
\end{equation}
and the generating function $\sum_{n=1}^\infty p_n x^n$, where $p_n=\#\Pairs_n$, is
\begin{equation}\label{eq:generating_s}
\left(\frac{1}{x}-\frac{1}{\sqrt{1-4x}}-1\right)^{-1}.
\end{equation}
\end{proposition}

For the proof of Proposition \ref{propos:smooth_recursion} we introduce additional notation. Let
$$
\Pi_n:=\left\{((i_0,i_1,\ldots,i_k),(j_1,\ldots,j_k)):\, 
\begin{aligned}
& k\geq 0,\,i_l\leq j_l \text{ for every } 1\leq l\leq k,\\
& 1=i_0<i_1<\cdots<i_k \text{ and } j_1<\cdots<j_k<n 
\end{aligned}
\right\}.
$$
For every $\pi:=((i_0,\ldots,i_k),(j_1,\ldots,j_k))\in \Pi_n$, define an endofunction $f_{\pi}$ on $[n]$ by
\[
f_\pi(i)=\begin{cases}j_l&\text{if }i_{l-1}\leq i<i_l\text{ for }l\in [k],\\
n&\text{if }i\geq i_k.\end{cases}
\]
The map $\pi\mapsto f_{\pi}$ is a bijection $\Pi_n\to\tilde\Dyck_n$.
In particular, for every $n\geq 1$,
\begin{equation}\label{eq:Pi}
\#\Pi_n=C_{n-1}.
\end{equation}
Moreover, for every $\pi:=((i_0,\ldots,i_k),(j_1,\ldots,j_k))\in \Pi_n$, let
$$I_{\pi}:=\{l\in [k]: \#(\{i_1,\ldots,i_k\}\cap[j_l])=l\}.$$
Note that if $k\geq 1$, then $k$ necessarily belongs to $I_{\pi}$.
Observe that $f_{\pi}(i)<f_{\pi}(f_{\pi}(i))=f_{\pi}(i+1)$ if and only if $i=i_l-1$ for $ l\in I_{\pi}$.
Hence, by \eqref{eq:g_counting},
$$\#\{g:[n]\to \{0,1\}: (f_{\pi},g)\in \Pairs_n\}=2^{\#I_{\pi}}.$$
It follows that for every positive integer $n$,
\begin{equation}\label{eq:brackets}
\tilde p_n=\sum_{\pi\in\Pi_n}2^{\# I_{\pi}}.
\end{equation}

\begin{proof}[Proof of Proposition \ref{propos:smooth_recursion}]
First note that
$$\{((i_0,\ldots,i_k),(j_1,\ldots,j_k))\in\Pi_n: k=0 \text{ or } j_k<n-1\}=\Pi_{n-1}.$$
For every $\pi:=((i_0,\ldots,i_k),(j_1,\ldots,j_k))\in\Pi_n\setminus\Pi_{n-1}$, let
\begin{align*}
l_{\pi}&:=\max\left\{l\in [k]: \#(\{j_1,\ldots,j_k\}\cap[i_l-1])=l-1\right\},\\
L(\pi)&:=((i_0,\ldots,i_{l_{\pi}-1}),(j_1,\ldots,j_{l_{\pi}-1})),\\
R(\pi)&:=((i_{l_{\pi}}-i,\ldots i_k-i),(j_{l_{\pi}}-i,\ldots, j_{k-1}-i)).
\end{align*}
For every $i\in [n-2]$ let
$$\Pi_{n,i}:=\{\pi\in\Pi_n\setminus\Pi_{n-1}: i_{l_{\pi}}=i\}.$$
The map $\pi\mapsto (L(\pi),R(\pi))$
is a bijection of $\Pi_{n,i}$ to $\Pi_{i+1}\times \Pi_{n-i-1}$ and for every
$\pi=((i_0,\ldots,i_k),(j_1,\ldots,j_k))\in \Pi_{n,i}$ we have
$$I_{L(\pi)}=I_{\pi}\setminus\{k\}.$$
It follows that
\begin{align*}
\tilde p_n&=\sum_{\pi\in\Pi_n}2^{\#I_{\pi}}=\sum_{\pi\in\Pi_{n-1}}2^{\#I_{\pi}}+
\sum_{i=1}^{n-2}\Big(\sum_{\pi\in\Pi_{n,i}}2^{\#I_{\pi}}\Big)\\
&=\tilde p_{n-1}+\sum_{i=1}^{n-2}\Big(\sum_{(\pi_1,\pi_2)\in\Pi_{i+1}\times\Pi_{n-i-1}}2^{\#I_{\pi_1}+1}\Big)\\
&=\tilde p_{n-1}+2\sum_{i=1}^{n-2}\#\Pi_{n-i-1}\Big(\sum_{\pi_1\in\Pi_{i+1}}2^{\#I_{\pi_1}}\Big)=\tilde p_{n-1}+
2\sum_{i=1}^{n-2}C_{n-i-2}\tilde p_{i+1},
\end{align*}
which proves \eqref{eq:recursion_a}.
Let us denote the generating functions
\[
C(x)=\sum_{n=0}^{\infty}C_nx^n, \quad P(x)=\sum_{n=1}^{\infty}p_nx^n,\quad \tilde{P}(x)=\sum_{n=1}^{\infty}\tilde p_nx^n.
\]
The recurrence relation \eqref{eq:recursion_a} yields that
\[
\tilde{P}(x)-x=x\tilde{P}(x)+2x\,C(x)(\tilde{P}(x)-x)=\left(x+2x\,C(x)\right)(\tilde{P}(x)-x)+x^2.
\]
Therefore, since $C(x)=\frac{1-\sqrt{1-4x}}{2x}$, $\tilde{P}(x)$ is equal to
\[
x+\frac{x^2}{1-x-2x\,C(x)}=x+\frac{x^2}{\sqrt{1-4x}-x}=
\frac{x\sqrt{1-4x}}{\sqrt{1-4x}-x}=\frac{1}{\frac{1}{x}-\frac{1}{\sqrt{1-4x}}},
\]
which proves \eqref{eq:generating_a}.
Finally, for every $ n\geq 2$, obviously
\[
p_n=\tilde p_n+\sum_{i=1}^{n-1}\tilde p_i p_{n-i}.
\]
Therefore,
\[
P(x)=\tilde{P}(x)+\tilde{P}(x)P(x),
\]
hence, by \eqref{eq:generating_a},
\[
\frac{1}{P(x)}=\frac{1}{\tilde{P}(x)}-1=\frac{1}{x}-\frac{1}{\sqrt{1-4x}}-1,
\]
and \eqref{eq:generating_s} follows.
\end{proof}

\begin{corollary}[cf.~\cite{MR1626487, MR2376109, MR3645580} and the references therein]
Let $a_n$, $n\geq1$ be the number of smooth indecomposable permutations in $S_n$. Then,
\[
\sum_{n=1}^{\infty}a_n x^n=\left(\frac{1}{x}-\frac{1}{\sqrt{1-4x}}\right)^{-1}
\]
and
\[
\sum_{n=1}^{\infty}\#(S_n)_\smth x^n=\left(\frac1x-\frac1{\sqrt{1-4x}}-1\right)^{-1}.
\]
\end{corollary}

\section{From covexillary to smooth}\label{sec:retract}

In this section we prove Theorem \ref{thm: idemp}.

\subsection{}
The following observation follows directly from Observation \ref{obs:sub_nsub}.
\begin{observation}\label{obs:312}
Let $\tau\in S_n$ and $i<j$ be such that $\tau(i)<\tau(j)$ and let $\tau'=\tau T_{i,j}$. Then,
\begin{enumerate}
\item If there is $k>j$ for which $\tau(k)<\tau(i)$, then $\maxperm{{\tau'}^{-1}}\equiv\maxperm{\tau^{-1}}$.
\item If there is $k<i$ for which $\tau(k)>\tau(j)$, then $\maxperm{\tau'}\equiv\maxperm{\tau}$.
\end{enumerate}
\end{observation}

\begin{corollary} \label{cor: taujk}
Let $\tau\in S_n$ and suppose that $i<j<k<l$ and $\tau(l)<\tau(j)<\tau(k)<\tau(i)$.
Then, $\tbl(\tau T_{j,k})=\tbl(\tau)$.
\end{corollary}

\begin{lemma} \label{lem: tauTjk}
A permutation $\sigma$ is {\dbi} if and only if it satisfies the following property
\begin{equation}\label{eq:property}
\begin{aligned}
&\text{for any }\tau\leq\sigma \text{ and } i<j<k<l \text{ such that } \tau(l)<\tau(j)<\tau(k)<\tau(i)\\
&\text{we have } \tau T_{j,k}\leq\sigma.
\end{aligned}
\end{equation}
\end{lemma}

\begin{proof}
We first show that if $\sigma\in S_n$ satisfies the property \eqref{eq:property} and $\pi\in S_m$ appears as a pattern in $\sigma$, i.e., there are strictly increasing functions $\lambda,\eta:[m]\to[n]$ such that $\sigma\circ\lambda=\eta\circ\pi$, then $\pi$ also satisfies the property \eqref{eq:property}.

For every $\tau\in S_m$, define $\hat{\tau}\in S_n$ by $\hat{\tau}\circ\lambda=\eta\circ\tau$ on $[m]$
and $\hat{\tau}\equiv\sigma$ outside $\lambda([m])$.
It is easy to verify that for every $\tau\in S_m$ we have $\hat{\tau}\leq\sigma$ if and only if $\tau\leq\pi$.
Suppose now that $\tau\leq\pi$, $1\leq i<j<k<l\leq m$ and $ \tau(l)<\tau(j)<\tau(k)<\tau(i)$.
Let $\tau'=\tau T_{j,k}$.
Then, $\hat{\tau}\leq\sigma$, $\lambda(i)<\lambda(j)<\lambda(k)<\lambda(l)$ and
$\eta(\tau(i))<\eta(\tau(j))<\eta(\tau(k))<\eta(\tau(l))$, i.e.,
$\hat{\tau}(\lambda(i))<\hat{\tau}(\lambda(j))<\hat{\tau}(\lambda(k))<\hat{\tau}(\lambda(l))$.
Therefore, since $\sigma$ satisfies the property \eqref{eq:property},
$\widehat{\tau'}=\hat{\tau}T_{\lambda(j),\lambda(k)}\leq\sigma$ and hence $\tau'\leq\pi$ as required.

Thus, in order to show that every permutation that satisfies the property \eqref{eq:property} is {\dbi},
it is enough to check that this property is not satisfied for the four permutations
$(4231)$, $(35142)$, $(42513)$ and $(351624)$, for which we can take
$\tau=(\underline{4231})$, $(1\underline{5342})$, $(\underline{4231}5)$ and $(1\underline{5342}6)$ respectively
where we underlined the entries with indices $i<j<k<l$.

Conversely, suppose that $\sigma$ is {\dbi} and let $\tau$ and $i<j<k<l$ be as in \eqref{eq:property}.
Since $\tau\leq\sigma$, $\maxperm{\tau}\leq\maxperm{\sigma}$ and $\maxperm{\tau^{-1}}\leq\maxperm{\sigma^{-1}}$ pointwise.
Let $\tau'=\tau T_{j,k}$. By Observation \ref{obs:312}, $\maxperm{\tau'}\equiv\maxperm{\tau}$ and
$\maxperm{{\tau'}^{-1}}\equiv\maxperm{\tau^{-1}}$ and hence, $\maxperm{\tau'}\leq\maxperm{\sigma}$ and
$\maxperm{{\tau'}^{-1}}\leq\maxperm{\sigma^{-1}}$.
Therefore, since $\sigma$ is {\dbi}, $\tau'\leq\sigma$.
\end{proof}

\subsection{}
We need another result.
\begin{lemma} \label{lem: covexup}
Suppose that $\tau\in S_n$ is covexillary but not smooth. Then, there exist $i<j<k<l$
such that $\tau(l)<\tau(j)<\tau(k)<\tau(i)$ and $\tau T_{j,k}$ is covexillary.
More precisely, suppose that $i<l$ is such that the set
$$P:=\{(j,k): i<j<k<l\text{ and }\tau(l)<\tau(j)<\tau(k)<\tau(i)\}$$
is non-empty. Then, $\exists (j,k)\in P$ such that $\tau T_{j,k}$ is covexillary.
\end{lemma}

The (rather technical) proof will be given in several steps.
For the rest of the subsection, fix a covexillary $\tau$ in $S_n$.

Denote
\begin{align*}
A_0:=&\{a<i: \tau(l)<\tau(a)<\tau(i)\},\\
A_1:=&\{a>l : \tau(l)<\tau(a)<\tau(i)\},\\
B_0:=&\{b<\tau(l) : i<\tau^{-1}(b)<l\},\\
B_1:=&\{b>\tau(i) : i<\tau^{-1}(b)<l\},\\
P_0:=&\{(j,k)\in P:\tau(j)>\tau(a) \text{ for every }a\in A_0\},\\
P_1:=&\{(j,k)\in P:\tau(k)<\tau(a) \text{ for every }a\in A_1\}.
\end{align*}
Note that since $\tau$ is covexillary, at most one of the sets $A_0,B_0$ is non-empty and at most one of the sets $A_1,B_1$ is non-empty.

Lemma \ref{lem: covexup} will easily follow from the following claim which will be proved below.

\begin{claim}\label{claim3}~
\begin{enumerate}
\item Suppose that $B_1=\emptyset\ne P_0$.
Let $(j,k)$ be the minimal element of $P_0$ with respect to the lexicographic order from left to right.
Then, $\tau':=\tau T_{j,k}$ is covexillary.
\item Similarly, if $B_0=\emptyset\ne P_1$, let $(j,k)$ be the maximal element of $P_1$ with respect to the lexicographic order
from right to left.
Then, $\tau':=\tau T_{j,k}$ is covexillary.
\item Suppose that $B_0=B_1=P_0=P_1=\emptyset$.
Let $(j,k)\in P$ be such that $j$ is maximal and $k$ is minimal (for that $j$).
Then, $\tau':=\tau T_{j,k}$ is covexillary.
\end{enumerate}
\end{claim}

\begin{proof}[Proof of Lemma \ref{lem: covexup}]
Passing to $\tau^{-1}$ if necessary, we may assume that $B_1=\emptyset$.
If $P_0\neq\emptyset$, then we can invoke the first part of Claim \ref{claim3}.
Therefore we may assume that $P_0=\emptyset$.
In particular, $A_0\neq\emptyset$ and hence $B_0=\emptyset$.
If $P_1\neq\emptyset$, then we are done by the second part of Claim \ref{claim3}.
Otherwise $P_1=\emptyset$ as well, and we apply third part of Claim \ref{claim3}.
\end{proof}

Before proving Claim \ref{claim3} we need another fact.

\begin{claim}\label{claim1}
Suppose that $(j,k)\in P$, $a<b<c<d$ and
\begin{equation}\label{eq:tau'}
\tau'(c)<\tau'(d)<\tau'(a)<\tau'(b)
\end{equation}
where $\tau'=\tau T_{j,k}$. Then,
\begin{enumerate}
\item $j\in\{a,b\}$ or $k\in \{c,d\}$.
\item
\begin{itemize}
\item If $a=j$, then $b<k$ and $\tau(d)>\tau(j)$.
\item If $b=j$, then $c\leq k$ and $\tau(a)>\tau(j)$.
\item If $c=k$, then $b\geq j$ and $\tau(d)<\tau(k)$.
\item If $d=k$, then $c>j$ and $\tau(a)<\tau(k)$.
\end{itemize}
\item
Suppose that $B_1=\emptyset$ and that $a=j$ or $c=k>j\neq b$. Then, $(j,b)\in P$ and $b<k$.
Similarly, suppose that $B_0=\emptyset$ and that $b=j<k\neq c$ or $d=k$. Then, $(c,k)\in P$ and $c>j$.
\end{enumerate}
\end{claim}

\begin{proof}
Observe first that if $x<y$ and $T_{j,k}(x)>T_{j,k}(y)$ then $x=j$ or $y=k$.

Since $\tau(T_{j,k}(c))<\tau(T_{j,k}(d))<\tau(T_{j,k}(a))<\tau(T_{j,k}(b))$ by \eqref{eq:tau'} and $\tau$ is covexillary,
we cannot have $ T_{j,k}(a)< T_{j,k}(b)< T_{j,k}(c)< T_{j,k}(d)$.
Therefore, $T_{j,k}(a)> T_{j,k}(b)$, $ T_{j,k}(b)> T_{j,k}(c)$ or $ T_{j,k}(c)> T_{j,k}(d)$.

If $ T_{j,k}(b)> T_{j,k}(c)$ then $b=j$ or $c=k$, by the observation above.
Suppose that $T_{j,k}(a)> T_{j,k}(b)$.
If $a\neq j$ then necessarily $b=k$, by the observation above.
It follows that $\tau(a)=\tau'(a)$, $\tau(b)=\tau(k)>\tau(j)=\tau'(b)$, $\tau(c)=\tau'(c)$ and $\tau(d)=\tau'(d)$.
Hence, by \eqref{eq:tau'}, $\tau(c)<\tau(d)<\tau(a)<\tau(b)$, contradicting the covexillarity of $\tau$, since $a<b<c<d$.
Therefore, $a=j$.
Similarly, by applying the same argument to $w_0\tau w_0$, we get that if $ T_{j,k}(c)> T_{j,k}(d)$, then $d=k$.
This completes the proof of the first part.

Suppose that $a=j$.
Then $\tau'(b)>\tau'(a)=\tau(k)>\tau(j)=\tau'(k)$ and hence $b\neq k$. In particular, $\tau(b)=\tau'(b)$.
Assume that $b>k$. Then, it follows that $\tau(c)=\tau'(c)$ and $\tau(d)=\tau'(d)$.
Therefore, by \eqref{eq:tau'} we get that $\tau(c)<\tau(d)<\tau(k)<\tau(b)$, refuting the covexillarity of $\tau$, since $k<b<c<d$.
Hence, $b<k$.
Assume that $\tau(d)<\tau(j)$.
In particular, $d\neq k$ and hence $\tau(d)=\tau'(d)$.
Therefore, $\tau'(c)<\tau'(d)=\tau(d)<\tau(j)=\tau'(k)$, thus $c\neq k$ and hence $\tau(c)=\tau'(c)$.
Therefore, by \eqref{eq:tau'}, $\tau(c)<\tau(d)<\tau(j)<\tau(b)$, rebutting the covexillarity of $\tau$, since $j=a<b<c<d$.
Hence, $\tau(d)>\tau(j)$.

Similarly, by applying the same argument to $w_0\tau w_0$, we get that if $d=k$ then $c>j$ and $\tau(a)<\tau(k)$.

Suppose that $b=j$.
Then $\tau(a)=\tau'(a)$ and $\tau(k)=\tau'(b)$.
If $c>k$, it follows that $\tau(c)=\tau'(c)$ and $\tau(d)=\tau'(d)$ and we get by \eqref{eq:tau'} that $\tau(c)<\tau(d)<\tau(a)<\tau(k)$,
denying the covexillarity of $\tau$, since $a<k<c<d$.
Therefore, $c\leq k$.
Assume that $\tau(a)<\tau(j)$.
Then, $\tau'(c)<\tau'(d)<\tau'(a)=\tau(a)<\tau(j)<\tau(k)$, and hence $\tau(c)=\tau'(c)$ and $\tau(d)=\tau'(d)$.
Therefore, we get by \eqref{eq:tau'} that $\tau(c)<\tau(d)<\tau(a)<\tau(j)$, gainsaying the covexillarity of $\tau$, since $a<j=b<c<d$.
Hence, $\tau(a)>\tau(j)$.

Similarly, by applying the same argument to $w_0\tau w_0$, we get that if $c=k$, then $j\leq b$ and $\tau(d)<\tau(k)$.

Finally, we prove the last part.
Suppose that $B_1=\emptyset$ and that $a=j$ or $c=k>j\neq b$.
Using the second part we get that $j<b<k$, and hence $\tau(b)<\tau(i)$, since $B_1=\emptyset$, and $\tau(b)=\tau'(b)$.
Hence, $\tau(b)>\tau(j)$ since $\tau'(b)>\tau'(a)=\tau(k)>\tau(j)$ if $a=j$ and $\tau'(b)>\tau'(c)=\tau(j)$ if $c=k$.
It follows that $(j,b)\in P$.
Similarly, by applying the same argument to $w_0\tau w_0$, we get the last assertion.
\end{proof}

\begin{proof}[Proof of Claim \ref{claim3}]
To prove the first part, suppose on the contrary that there are $a<b<c<d$ such that
$\tau'(c)<\tau'(d)<\tau'(a)<\tau'(b)$.

If $a=j$ or $c=k>j\neq b$, then by the third part of Claim \ref{claim1}, $(j,b)\in P$ and hence $(j,b)\in P_0$, and $b<k$,
contradicting the minimality of $k$.

Therefore, according to the first part of Claim \ref{claim1}, we may assume that $b=j$ or $d=k>j\neq a$.
If $b=j$, then $\tau(a)>\tau(j)$ by the second part of Claim \ref{claim1}, and $\tau(a)=\tau'(a)<\tau'(b)=\tau(k)$.
If $d=k$, then $\tau(a)<\tau(k)$ by the second part of Claim \ref{claim1}.
If additionally $j\neq a$, then $\tau(a)=\tau'(a)>\tau'(d)=\tau(j)$.
Moreover, $a<j$, otherwise $(j,a)\in P$ and hence $(j,a)\in P_0$, violating the minimality of $k$.
In any case, $a<j$ and $\tau(j)<\tau(a)<\tau(k)$.
On the other hand, since $(j,k)\in P_0$, we have $\tau(j)>\tau(a_0)$ for every $a_0\in A_0$.
It follows that $a\notin A_0$ and hence necessarily $i<a$.
It also follows that $\tau(a)>\tau(a_0)$ for every $a_0\in A_0$. Therefore, $(a,k)\in P_0$, refuting the minimality of $j$.

The second part follows from the first part by considering $w_0\tau w_0$.

We turn to the third part.
The set $A_0$ is non-empty, since $P_0=\emptyset\neq P$.
Let $a_0\in A_0$ be such that $\tau(a_0)$ is maximal.
Similarly, let $a_1\in A_1$ be such that $\tau(a_1)$ is minimal.
Note that $\tau(a_0)<\tau(a_1)$, otherwise $a_0<i<l<a_1$ and $\tau(l)<\tau(a_1)<\tau(a_0)<\tau(i)$, contradicting the covexillarity of $\tau$.
Also, $\tau(j)<\tau(a_0)$, since $(j,k)\notin P_0$. Similarly, $\tau(a_1)<\tau(k)$.
In a way of contradiction, suppose that there are $a<b<c<d$ such that $\tau'(c)<\tau'(d)<\tau'(a)<\tau'(b)$.

If $a=j$ or $c=k>j\neq b$, then $(j,b)\in P$ and $b<k$, by the third part of Claim \ref{claim1},
disproving the minimality of $k$.
Similarly, if $b=j<k\neq c$ or $d=k$, then $(c,k)\in P$ and $c>j$, by the third part of Claim \ref{claim1},
invalidating the maximality of $j$.

Therefore, according to the first part of Claim \ref{claim1}, we may assume that $b=j$ and $c=k$.
Then, $\tau(a)=\tau'(a)<\tau'(b)=\tau(k)$ and $\tau(j)=\tau'(c)<\tau'(d)=\tau(d)$.
Therefore $i<a$, otherwise $a<i<j<d$ and $\tau(j)<\tau(d)<\tau(a)<\tau(i)$, rebuffing the covexillarity of $\tau$.
Moreover, $\tau(a_0)<\tau(d)$ and hence $\tau(a_0)<\tau(a)$, otherwise $a_0<i<j<d$ and $\tau(j)<\tau(d)<\tau(a_0)<\tau(i)$, contradicting the covexillarity of $\tau$.
Therefore $(a,k)\in P_0$, denying the emptiness of $P_0$.
\end{proof}

\subsection{}
We can now prove Theorem \ref{thm: idemp}.

If $\tau$ is covexillary, then the set $\tbl(\tau)$ is {\adm} by Lemma \ref{lem: covexisadm}
and hence the permutation $\pi(\tbl(\tau))$ is smooth by Proposition \ref{propos:bijection}.

The map $\tau\mapsto\pi(\tbl(\tau))$ from the set of covexillary permutations to the set of smooth permutations is an idempotent function, since $\pi(\tbl(\sigma))=\sigma$ for any smooth $\sigma$, by Proposition \ref{propos:order}.
Next we show that this map is order preserving.
First note that if $\tau$ is covexillary then $\sigma:=\pi(\tbl(\tau))$ is covexillary as well (since it is smooth)
and $\tbl(\sigma)=\tbl(\tau)$ by Proposition \ref{propos:bijection}, hence $\tblc(\sigma)=\tblc(\tau)$ by Corollary \ref{cor:tbl_tblc}.
Suppose now that $\tau_1\leq\tau_2$ are covexillary permutations, and let $\sigma_1:=\pi(\tbl(\tau_1))$,  $\sigma_2:=\pi(\tbl(\tau_2))$.
Then, $\tblc(\sigma_1)=\tblc(\tau_1)\subseteq\tblc(\tau_2)=\tblc(\sigma_2)$ and hence $\sigma_1\leq \sigma_2$ by Lemma \ref{lem:tblc_dbi},
as $\sigma_2$ is {\dbi} (since it is smooth).

It follows that if $\tau$ is covexillary, then $\sigma=\pi(\tbl(\sigma))\geq\pi(\tbl(\tau))$ for every smooth $\sigma\geq\tau$. 
Hence, $\pi(\tbl(\tau))=\min\{\sigma\in S_n\text{ smooth}:\sigma\geq\tau\}$, since $\pi(\tbl(\tau))\geq\tau$ by \eqref{eq: pialt}.
\qed

\begin{remark}
For a general $\tau\in S_n$ there does not exist a smooth permutation $\sigma\geq\tau$ such that
$\tbl_\Trans(\sigma)=\tbl_\Trans(\tau)$, let alone $\tbl(\sigma)=\tbl(\tau)$. For instance for $\tau=(462513)\in S_6$,
the only smooth permutations $\sigma$ such that $\tbl_\Trans(\sigma)=\tbl_\Trans(\tau)$ are
$(654123)$ and $(456321)$ and none of them is $\geq\tau$.
\end{remark}

\begin{remark}
For $\tau=(3412)\in S_4$, the set $\{\sigma\in (S_n)_\smth:\sigma\geq\tau\}$ contains two minimal elements (namely $(4312)$ and $(3421)$).
Thus, the assumption on $\tau$ in Theorem \ref{thm: idemp} is essential.

In fact, it is not enough to require that $\tbl(\tau)$ is \adm.
Indeed, if $\tau=(345612)\in S_6$ and $\sigma=(654312)\in(S_6)_{\smth}$, then $\tbl(\tau)$ is {\adm} and $\tau\leq\sigma$ but
$\pi(\tbl(\tau))=(345621)\not\leq\sigma$.
\end{remark}

\section{Relation to coessential set} \label{sec: ess}

In \cite{MR1154177} Fulton introduced the notion of the \emph{essential set} of a permutation $\sigma\in S_n$.
For our purpose it is more convenient to use a slight variant, namely
$$
\Ess(\sigma)=\\\{(i,j)\in [n-1]^2:\sigma(i)\le j<\sigma(i+1)\text{ and }\sigma^{-1}(j)\le i<\sigma^{-1}(j+1)\}.
$$
In the notation of [ibid.,(3.8)] we have
\[
\Ess(\sigma)=\{(n-i,j):(i,j)\in Ess(\sigma w_0)\}=\{(i,n-j):(i,j)\in Ess(w_0\sigma)\}.
\]
For any $\sigma\in S_n$ we have
\begin{equation} \label{eq: essetprop}
\begin{aligned}
&\text{for any }\tau\in S_n,\\
&\tau\le\sigma\iff\#(\tau([i])\cap[j])\geq\#(\sigma([i])\cap[j])\text{ for all }(i,j)\in\Ess(\sigma).
\end{aligned}
\end{equation}
In particular, $\sigma$ is defined by the set $\Ess(\sigma)$ and the restriction of the function $\#(\sigma([i])\cap[j])$
to $\Ess(\sigma)$. The image of the injective map
\[
\sigma\in S_n\mapsto (\Ess(\sigma),\#(\sigma([i])\cap[j])\|_{(i,j)\in\Ess(\sigma)})
\]
was described in \cite{MR1412437}.

The set $\Ess(\sigma)$ is minimal with respect to the property \eqref{eq: essetprop}.
In other words, if we replace $\Ess(\sigma)$ by a proper subset then \eqref{eq: essetprop} will not hold.
In particular, $\sigma$ is defined by inclusion if and only if
$\#(\sigma([i])\cap[j])=\min(i,j)$ for all $(i,j)\in\Ess(\sigma)$.
In general, consider the subset
\[
\Ess^\circ(\sigma)=\{(i,j)\in\Ess(\sigma):\sigma([i])\subseteq[j]\text{ or }\sigma^{-1}([j])\subseteq[i]\}.
\]
Thus, $\sigma$ is defined by inclusion if and only if $\Ess(\sigma)=\Ess^\circ(\sigma)$, in which case $\sigma$ is determined by
the set $\Ess(\sigma)$. In particular, this is the case if $\sigma$ is smooth.

\begin{observation}\label{obs: essprop1}
For any $\sigma\in S_n$, if $(i_1,j_1)$ and $(i_2,j_2)$ are two distinct points
in $\Ess^\circ(\sigma)$ such that $j_2\ge j_1$ and $i_2\ge i_1$, then
$\max(i_2,j_2)>\max(i_1,j_1)$ and $\min(i_2,j_2)>\min(i_1,j_1)$.
\end{observation}

Let $\SEss$ be the set of subsets $E$ of $[n-1]^2$ such that for every two distinct points
$(i_1,j_1)$ and $(i_2,j_2)$ in $E$ such that $\min(i_2,j_2)\ge\min(i_1,j_1)$ we have
\[
i_2\ge i_1,\ \ j_2\ge j_1,\ \ \max(i_2,j_2)>\max(i_1,j_1)\text{ and }\min(i_2,j_2)>\min(i_1,j_1).
\]

\begin{lemma}
For every covexillary $\sigma\in S_n$ we have $\Ess^\circ(\sigma)\in\SEss$.
\end{lemma}

\begin{proof}
By Observation \ref{obs: essprop1} it is enough to show that there are no pairs $(i_1,j_1),(i_2,j_2)\in\Ess(\sigma)$
such that $i_1<i_2$ and $j_1>j_2$.
Assume on the contrary that this is not the case.
Then $\sigma^{-1}(j_1)<i_1+1\le i_2<\sigma^{-1}(j_2)$ and $\sigma(i_2)<j_2<j_1<\sigma(i_1+1)$
in contradiction to the assumption that $\sigma$ is covexillary.
\end{proof}

For any $(f,g)\in\Pairs_n$ let
\[
\begin{aligned}
E(f,g)=&\{(i,f(i)):i\in[n-1],\ f(i+1)>f(i)\text{ and }g(i)=1\}\cup\\
&\{(f(i),i):i\in[n-1],\ f(i+1)>f(i)\text{ and }g(i)=0\}.
\end{aligned}
\]

For $E\in\SEss$ define $\hat{f}_E:[n]\rightarrow[n]$ by
$$
\hat{f}_E(k)=\min\Big(\{n\}\cup\left\{\max(i,j):(i,j)\in E\cap [k,n)^2\right\}\Big).
$$
\begin{observation}\label{obs:hat}
Suppose that $E\in\SEss$.
\begin{enumerate}
\item If $i<j$, then $|E\cap\{(i,j),(j,i)\}|\leq 1$.
\item If $i<j$ and $E\cap\{(i,j),(j,i)\}\neq\emptyset$, then $\hat{f}_E(i)=j<\hat{f}_E(i+1)$
\item If $\hat{f}_E(i)<\hat{f}_E(i+1)$, then $\{(i,\hat{f}_E(i)),(\hat{f}_E(i),i)\}\cap E\neq\emptyset$.
\end{enumerate}
\end{observation}
Let $\Delta:=\{(i,j):1\leq i<j<n\}$ and let
$$
\hat{g}_E(i)=\begin{cases}1&\text{if }(j,\hat{f}_E(j))\in E\cap\Delta \text{, where } j:=\max \hat{f}_E^{-1}(\hat{f}_E(i)),\\
0&\text{otherwise.}\end{cases}
$$

\begin{lemma}\label{lem:PEbijection}
The map $(f,g)\mapsto E(f,g)$ is a bijection $\Pairs_n\to \SEss_n$. The inverse map is
$E\mapsto (\hat{f}_E,\hat{g}_E)$.
\end{lemma}

\begin{proof}
Suppose that $(f,g)\in\Pairs$. We first show that $E(f,g)\in\SEss$.
Let $(i_1,j_1)$ and $(i_2,j_2)$ be two distinct points in $E(f,g)$ and assume that
\[
k_2:=\min(i_2,j_2)\ge k_1:=\min(i_1,j_1).
\]
Then $k_r\in[n-1]$, $f(k_r+1)>f(k_r)$ and $f(k_r)=\max(i_r,j_r)$, $r=1,2$.
Also, $k_1\ne k_2$. Therefore, $k_1<k_2$ and hence $f(k_1)<f(k_1+1)\le f(k_2)$.
If $f(k_1)\le k_2$ then clearly $i_1\le i_2$ and $j_1\le j_2$.
On the other hand, if $f(k_1)>k_2$ then $g(k_1)=g(k_2)$ (since $(f,g)\in\Pairs$)
and therefore either $i_r=k_r$ and $j_r=f(k_r)$, $r=1,2$ or $i_r=f(k_r)$ and $j_r=k_r$, $r=1,2$.
In both cases $i_1\le i_2$ and $j_1\le j_2$.

For every $k$, clearly,
$$
\hat{f}_{E(f,g)}(k)=\min\Big(\{n\}\cup\left\{f(i): k\leq i<n\text{ and }f(i+1)>f(i)\right\}\Big)=f(k).
$$

Moreover, if $j=\max f^{-1}(f(i))$, then $f(j)=f(i)$, $g(j)=g(i)$ and if $j<n$ then $f(j+1)>f(j)$.
Therefore, $\hat{g}_{E(f,g)}(i)=1\iff (j,f(j))\in E(f,g)\cap\Delta \iff g(j)=1 \iff g(i)=1$. Hence, $(\hat{f}_{E(f,g)},\hat{g}_{E(f,g)})=(f,g)$.

Suppose now that $E\in\SEss$. We show that $(\hat{f}_E,\hat{g}_E)\in\Pairs$. It is clear that $\hat{f}_E\in\Dyck_n$.
If $\hat{f}_E(\hat{f}_E(i))=\hat{f}_E(i)$ then clearly ${\hat{f}_E}^{-1}(\hat{f}_E(i))=\hat{f}_E(i)$ and hence $\hat{g}_E(i)=0$, since\\ $(\hat{f}_E(i),\hat{f}_E(\hat{f}_E(i)))=(\hat{f}_E(i),\hat{f}_E(i))\notin\Delta$ .

Suppose that $\hat{g}_E(i)\neq\hat{g}_E(i+1)$.
Then necessarily $\max{\hat{f}_E}^{-1}(\hat{f}_E(i))=i$, and let $j:=\max{\hat{f}_E}^{-1}(\hat{f}_E(i+1))$.
If $\hat{g}_E(i)=1$ and $\hat{g}_E(i+1)=0$ then $(i,\hat{f}_E(i))\in E$ and $(\hat{f}_E(j),j)\in E$.
Since $\in\SEss$ it follows that $\hat{f}_E(i)<j$ and hence $\hat{f}_E(\hat{f}_E(i))\leq\hat{f}_E(j)=\hat{f}_E(i+1)$.
Similarly if $\hat{g}_E(i)=0$ and $\hat{g}_E(i+1)=1$.

Finally we show that $E(\hat{f}_E,\hat{g}_E)=E$.
Suppose that $\hat{f}_E(i+1)>\hat{f}_E(i)$. Then, \\$\max{\hat{f}_E}^{-1}(\hat{f}_E(i))=i$ and hence, if $\hat{g}_E(i)=1$ then $(i,\hat{f}_E(i))\in E$ and if $\hat{g}_E(i)=0$ then $(i,\hat{f}_E(i))\notin E$ or $i=\hat{f}_E(i)$ and in any case, necessarily $(\hat{f}_E(i),i)\in E$ by Observation \ref{obs:hat}.
Conversely, suppose first that $(i,r)\in E\cap\Delta$.
Then $\hat{f}_E(i)=r<\hat{f}_E(i+1)$  by Observation \ref{obs:hat}.
Hence, ${\hat{f}_E}^{-1}(\hat{f}_E(i))={\hat{f}_E}^{-1}(r)=i$ and therefore $\hat{g}(i)=1$, since $(i,\hat{f}_E(i))=(i,r)\in E\cap\Delta$.
Hence, $(i,r)\in E(\hat{f}_E,\hat{g}_E)$.
Similarly, if $(r,i)\in E\setminus\Delta$ then $\hat{f}_E(i)=r<\hat{f}_E(i+1)$ by Observation \ref{obs:hat}, hence ${\hat{f}_E}^{-1}(\hat{f}_E(i))={\hat{f}_E}^{-1}(r)=i$ and therefore $\hat{g}(i)=0$, since $(i,\hat{f}_E(i))=(i,r)\notin\Delta$. Hence, $(r,i)\in E(\hat{f}_E,\hat{g}_E)$.
\end{proof}

\begin{observation}\label{obs: 23tabfromess}
For any $\sigma\in S_n$, the table $\tbl(\sigma)$ is determined by the set $\Ess^\circ(\sigma)$.
More precisely, we have
\begin{align*}
T_{i,j}\in\tbl(\sigma)&\iff \Ess^\circ(\sigma)\cap [i,j)^2=\emptyset\\
R_{i,j,k}\in\tbl(\sigma)&\iff\Ess^\circ(\sigma)\cap [i,j)^2=\Ess^\circ(\sigma)\cap [j,k)\times [i,k)=\emptyset\\
L_{i,j,k}\in\tbl(\sigma)&\iff\Ess^\circ(\sigma)\cap [i,j)\times [i,k)=\Ess^\circ(\sigma)\cap [j,k)^2=\emptyset.
\end{align*}
\end{observation}

For $E\subseteq [n-1]^2$, let
\begin{align*}
A_E=&\{T_{i,j}:E\cap [i,j)^2=\emptyset\}\cup\\
&\{R_{i,j,k}:E\cap [i,j)^2=E\cap [j,k)\times [i,k)=\emptyset\}\cup\\
&\{L_{i,j,k}:E\cap [i,j)\times [i,k)=E\cap [j,k)^2=\emptyset\}.
\end{align*}

\begin{proposition}
We have a commutative diagram of bijections
\[
\begin{tikzcd}
\Pairs_n\arrow[r, leftrightarrow] \arrow[rd, leftrightarrow] & \SEss_n \arrow[d, leftrightarrow]\\
(S_n)_{\smth} \arrow[r, leftrightarrow] \arrow[ru, leftrightarrow] \arrow[u, leftrightarrow] & \Adm_n
\end{tikzcd}
\]
The upper horizontal maps are $(f,g)\mapsto E(f,g)$ and $E\mapsto (\hat{f}_E,\hat{g}_E)$.
The right vertical maps are $E\mapsto A_E$ and $A\mapsto E(f_A,g_A)$.
The principal diagonal maps are $(f,g)\mapsto A_{f,g}$ and $A\mapsto (f_A,g_A)$.
The non-principal diagonal maps are $\sigma\mapsto\Ess^\circ(\sigma)$ and $E\mapsto \sigma(\hat{f}_E,\hat{g}_E)$.
The left vertical maps are $(f,g)\mapsto\sigma(f,g)$ and $\sigma\mapsto (f_{\tbl(\sigma)},g_{\tbl(\sigma)})$.
The lower horizontal maps are $\sigma\mapsto\tbl(\sigma)$ and $A\mapsto \pi(A)$.
\end{proposition}

\begin{proof}
Observation \ref{obs: 23tabfromess} readily yields that $A_{\Ess^\circ(\sigma)}=\tbl(\sigma)$ for every $\sigma$, and it is easy to verify that $A_{E(f,g)}=A_{f,g}$ for every $(f,g)\in\Pairs$.
Therefore, the proposition follows from Theorem \ref{thm:bijection}, Proposition \ref{prop:Afg_bijection}, Proposition \ref{propos:sigma_fg} and Lemma \ref{lem:PEbijection}.
\end{proof}

\begin{corollary}
\[
\SEss=\{\Ess^\circ(\sigma):\sigma\in (S_n)_{\smth}\}=
\{\Ess^\circ(\sigma):\sigma\in S_n\text{ covexillary}\}.
\]
\end{corollary}

\begin{corollary}
For any covexillary $\tau\in S_n$ we have $\Ess(\pi(\tbl(\tau)))=\Ess^\circ(\tau)$.
\end{corollary}

\end{document}